\documentclass[a4,11pt]{article}

\usepackage{graphicx}
\usepackage{amssymb}
\usepackage{amsthm}
\usepackage{amsmath}
\usepackage{mathrsfs}
\usepackage{float}
\usepackage{bm}

\def\Xint#1{\mathchoice
   {\XXint\displaystyle\textstyle{#1}}%
   {\XXint\textstyle\scriptstyle{#1}}%
   {\XXint\scriptstyle\scriptscriptstyle{#1}}%
   {\XXint\scriptscriptstyle\scriptscriptstyle{#1}}%
   \!\int}
\def\XXint#1#2#3{{\setbox0=\hbox{$#1{#2#3}{\int}$}
     \vcenter{\hbox{$#2#3$}}\kern-.5\wd0}}

\def\dashint{\Xint-}

\def \dsp {\displaystyle}

\newtheorem{theorem}{Theorem}[section]
\newtheorem{proposition}[theorem]{Proposition}

\newtheorem{example}{Example}[section] 
\newtheorem{remark}{Remark}[section]

\DeclareMathAlphabet{\mathpzc}{OT1}{pzc}{m}{it}

\newcommand{\T}{\mathbb{T}}
\newcommand{\PP}{\mathbb{P}}
\newcommand{\bl}{{\tiny\bm[\normalsize}}
\newcommand{\br}{{\tiny\bm]\normalsize}}

\renewcommand{\Re}{\mathop{\rm Re}}
\renewcommand{\Im}{\mathop{\rm Im}}

\begin{document}

\title{Quadrature rules from a $R_{II}$ type recurrence relation and associated quadrature rules on the unit circle\thanks{This work is part of the PhD thesis of the second author at UNESP and supported by a grant from CAPES of Brazil. The first author is partially supported by funds from CNPq (305208/2015-2,  402939/2016-6) of Brazil. The third author is partially supported by funds from FAPESP (2016/09906-0, 2017/12324-6) and CNPq (305073/2014-1) of Brazil.}}

\author
{
 {Cleonice F. Bracciali, \ Junior A. Pereira, \ A. Sri Ranga}
  \\[0.5ex]
 {\small $^1$Depto de Matem\'{a}tica Aplicada, IBILCE, } \\
 {\small UNESP - Univ Estadual Paulista,} 
 {\small 15054-000, S\~{a}o Jos\'{e} do Rio Preto, SP, Brazil.  } \\
 {\small cleonice.bracciali@unesp.br, \ junior.gusto@hotmail.com, \ sri.ranga@unesp.br}
  }

\maketitle

\begin{abstract}
We consider the theoretical and numerical aspects of the quadrature rules associated with a sequence of polynomials generated by a special $R_{II}$ recurrence relation. We also look into some  methods for generating the nodes (which lie on the real line) and the positive weights of these quadrature rules.  With a simple transformation these quadrature rules on the real line also lead to  certain positive quadrature rules of highest algebraic degree of precision on the unit circle. This way, we also introduce  new approaches  to evaluate the nodes and weights of these specific quadrature rules on the unit circle.  \\

{\bf Keyword:} Orthogonal polynomials on the unit circle; Quadrature rules; $R_{II}$ type  recurrence relation.
\end{abstract}

\section{Introduction}
\label{Sec-Intro}

Positive quadrature rules or quadrature rules with positive weights are very important in the numerical evaluation of integrals.  In general they are also quadrature rules with certain kind of highest degree of precision, and thus,  have  served as nice tools for resolving many mathematical problems (see, for example, \cite{Gautschi-1981} and \cite{Ismail-1989}). In this paper we will consider the theoretical and numerical aspects of the positive quadrature rules associated with a sequences of polynomials generated by a special $R_{II}$ recurrence relation which was studied recently in \cite{Ismail-Ranga}. We will look into the methods for generating the nodes (which lie on the real line) and the positive weights of these quadrature rules.  With a simple transformation these quadrature rules on the real line also lead to  certain positive quadrature rules of highest algebraic degree of precision on the unit circle (i.e., Gaussian type quadrature rules associated with para-orthogonal polynomials on the unit circle). In this way, we also introduce  new approaches  to evaluate the nodes and weights of these specific quadrature rules on the unit circle.  

A general $R_{II}$ type recurrence relation takes the form   
\begin{equation*}  \label{Eq-TTRR-general-R2-type}
P_{n+1}(x)=\sigma_{n+1}(x-v_{n+1})P_n(x)-u_{n+1}(x-a_n)(x-b_n)P_{n-1}(x), 
\end{equation*}
for $n \geq 1$, with $P_0(x)=1$ and $P_1(x)=\sigma_1(x-v_1),$  where $\{\sigma_n\}, \{v_n\}, \{u_n\}, \{a_n\}$ and $\{b_n\}$ are complex sequences. A systematic study of such recurrence relations started in the work of Ismail and Masson \cite{Ismail-Masson-JAT1995}. In \cite{Ismail-Masson-JAT1995} these recurrence relations were referred to as those associated with $R_{II}$ type continued fractions. However, in the work of Zhedanov \cite{Zhedanov-JAT1999}, where the associated generalized eigenvalue problems have been established, these recurrence relations were simply referred to as $R_{II}$ type recurrence relations. 

The special $R_{II}$ type recurrence relation that we will deal with in the present paper is 
\begin{equation} \label{Eq-TTRR-Special-R2-type}
P_{n+1}(x)=(x-c_{n+1})P_n(x)-d_{n+1}(x^2+1)P_{n-1}(x), \quad n \geq 1,
\end{equation}
with $P_0(x)=1$ and $P_1(x)=x-c_1$, where $\{c_n\}_{n\geq 1}$  and $\{d_{n+1}\}_{n\geq 1}$ are  real sequences and further $\{d_{n+1}\}_{n\geq 1}$ is a positive chain sequence. That is, there exists a sequence $\{\ell_{n+1}\}_{n \geq 0}$, where $\ell_1 = 0$ and $0 < \ell_{n+1} < 1$, $n \geq 1$, such that  $d_{n+1} = (1-\ell_{n})\ell_{n+1}$, $n \geq 1$. 

Any sequence $\{g_{n+1}\}_{n \geq 0}$, where $0 \leq g_1 < 1$ and $0 < g_{n+1} < 1$, $n \geq 1$, such that  $d_{n+1} = (1-g_{n})g_{n+1}$, $n \geq 1$, is called a parameter sequence of the positive chain sequence $\{d_{n+1}\}_{n\geq 1}$. The parameter sequence $\{\ell_{n+1}\}_{n \geq 0}$ is normally referred to as the minimal parameter sequence of $\{d_{n+1}\}_{n\geq 1}$. There are positive chain sequences that have only the minimal parameter sequence. If $\{d_{n+1}\}_{n\geq 1}$ is not such a positive chain sequence, then it has infinitely many parameter sequences $\{g_{n+1}\}_{n \geq 0}$.  However, there is one parameter sequence of $\{d_{n+1}\}_{n\geq 1}$, which we denote by $\{M_{n+1}\}_{n \geq 0}$, that is characterized by  $0 < g_{n} < M_{n}$, $n \geq 1$, and  
\[
    \sum_{n = 2}^{\infty} \prod_{k=2}^{n} \frac{M_{k}}{1-M_{k} }  \ = \infty. 
\]
The above infinite series formula is known as the Wall's criteria for maximal parameter sequence. For definitions and for many of the properties associated with positive chain sequences, we refer to Chihara \cite{Chihara-Book}. 

From \eqref{Eq-TTRR-Special-R2-type} one can verify that $P_n$ is a polynomial of exact degree $n$ and that the leading coefficient $\mathfrak{p}_{n}$ of $P_n$ is 
\begin{equation} \label{Eq-LeadCoeff-Pn}
    \mathfrak{p}_{n} = \prod_{k=1}^{n}(1-\ell_{k}). 
\end{equation}
The nodes of the quadrature rule on the real line that we study in the present manuscript are the zeros of the polynomials $P_{n}$.

The recurrence relation \eqref{Eq-TTRR-Special-R2-type} was the principal object of study in the recent paper  \cite{Ismail-Ranga}. It is known that the zeros of the polynomial $P_{n}$ obtained from \eqref{Eq-TTRR-Special-R2-type} are real and simple and they also interlace with the zeros of the polynomial $P_{n+1}$. That is, if we denote the zeros of $P_{n}$ by $x_{n,k}$, $k=1,2, \ldots, n$, where $x_{n,k} > x_{n, k+1}$, $1 < k < n-1$, then 
\begin{equation} \label{Eq-zero-arrangement}
     x_{n+1,1} > x_{n,1} > x_{n+1,2} > \cdots > x_{n+1,n} > x_{n,n} > x_{n+1,n+1}, \quad n \geq 1. 
\end{equation}

From \cite{Ismail-Ranga} it follows that $P_{n}$ is the characteristic polynomial of the generalized tridiagonal eigenvalue problem or tridiagonal pencil $(\mathbf{A}_n, \mathbf{B}_n)$ given by 
\begin{equation} \label{Eq-Special-GEV-problem}
\mathbf{A}_n \mathbf{u}_{n}=x \mathbf{B}_n \mathbf{u}_{n}, \quad n \geq 1,
\end{equation}
where the Hermitian tridiagonal matrices $\mathbf{A}_{n}$ and  $\mathbf{B}_{n}$
are, respectively,  
\[
    \left[ 
       \begin{array}{cccccc}
        c_1  & \! i\sqrt{d_2} & 0  & \!\!\! \cdots & 0 & 0  \\[1ex]
        \!\!\! -i\sqrt{d_2}  & c_2  & \!\! i\sqrt{d_3} & \!\!\! \cdots & 0 & 0 \\[1ex]
         0  & \!\!\!\!\! -i\sqrt{d_3}  & c_3 & \!\!\! \cdots & 0 & 0  \\[1ex]
         \vdots  & \vdots  & \vdots &    & \vdots & \vdots \\[1ex]
          0  & 0  & 0 & \!\!\! \cdots & c_{n-1} & \!\!\! i\sqrt{d_{n}}  \\[1ex]
          0  & 0  & 0 & \!\!\! \cdots & \!\!\!\!\!  -i\sqrt{d_{n}} & c_{n} 
      \end{array} \! \right], 
\quad 
    \left[ 
       \begin{array}{cccccc}
        1 & \!\sqrt{d_2} & 0 & \cdots & 0 & 0  \\[1ex]
        \!\sqrt{d_2} & 1 &  \!\sqrt{d_3} & \cdots & 0 & 0 \\[1ex]
         0 & \!\sqrt{d_3} & 1 & \cdots & 0 & 0  \\[1ex]
         \vdots & \vdots & \vdots &    & \vdots & \vdots \\[1ex]
          0 & 0 & 0 & \cdots & 1 & \!\sqrt{d_{n}}  \\[1ex]
          0 & 0 & 0 & \cdots & \!\!\!\sqrt{d_{n}} & 1 \\[1ex]
      \end{array} \right].
\] \\[0ex]
The sequence $\{d_{n+1}\}_{n \geq 1}$ being a positive chain sequence is also equivalent of saying that the matrices $\mathbf{B}_n$, $n \geq 1$ are all positive definite matrices. 

Therefore,  the problem of determining the zeros of $P_n$ (i.e., determining the nodes of the associated $n$-point quadrature rule) is also a problem of determining the eigenvalues of this generalized eigenvalue problem. To be precise, if $u_{0}(x) = P_0(x)$,  
\begin{equation} \label{Eq-Rationals-1}
      u_n(x) =  \frac{(-1)^{n}}{(x-i)^{n} \prod_{j=1}^{n}\sqrt{d_{j+1}}} P_n(x) , \quad n \geq 1,
\end{equation} 
and $\mathbf{u}_{n}(x) = \big[u_{n,0}(x), u_{n,1}(x), \ldots, u_{n,n-1}(x)\big]^{T}$, then $\mathbf{A}_n \mathbf{u}_{n}(x_{n,j}) = x_{n,j} \mathbf{B}_n \mathbf{u}_{n}(x_{n,j})$, $ n \geq 1$ and 
\begin{equation} \label{Eq-PosDef-formula}
 \begin{array}{l}
   \dsp   \frac{P_{n}^{\prime}(x) P_{n-1}(x) - P_{n-1}^{\prime}(x) P_{n}(x)}{(x^2+1)^{n-1} d_2 d_3 \cdots d_{n}}   = \mathbf{u}_n(x)^{H} \mathbf{B}_{n}   \mathbf{u}_n(x) \, > 0, \quad n \geq 2.
 \end{array} 
\end{equation}

Section \ref{Sec-Statements} of this paper, which brings the statements of our main objectives, starts by giving the necessary preliminary results and concepts that will be required for the development of the results obtained in the paper. Some of these preliminary results are composed as Theorem \ref{Thm-Basics}.   

Theorem \ref{Thm-QuadRule-RealLine} in Section \ref{Sec-Statements}, which is one of the main results in this paper, is with respect to the quadrature rules that follow from the $R_{II}$ type recurrence \eqref{Eq-TTRR-Special-R2-type}. That is, with respect to the quadrature rules that follow from the sequences $\{c_{n}\}_{n \geq 1}$ and $\{d_{n+1}\}_{n \geq 1}$. The proof of Theorem \ref{Thm-QuadRule-RealLine} is given in Section \ref{Sec-QuadRules-RL}.
Sequence of polynomials $\{Q_{n}\}_{n\geq 0}$ which satisfy the same three term recurrence as that of $\{P_{n}\}_{n\geq 0}$, but with the initial conditions $Q_{0}(x) = 0$ and $Q_{1}(x) = constant$,  also play an important role in Section \ref{Sec-QuadRules-RL}. 

The results that cover the  related quadrature rules on the unit circle, the other main objective of this paper, are presented in Theorem \ref{Thm-QuadRule-UnitCircle}. The proof of  Theorem \ref{Thm-QuadRule-UnitCircle} and other related results such as how to get the values of the coefficients $\{c_{n}\}_{n \geq 1}$ and $\{d_{n+1}\}_{n \geq 1}$ form a given measure on the unit circle are  given in Section \ref{Sec-QuadRules-UC}.  

In Section \ref{Sec-Eample-simple} we provide a simple example, where the polynomials and the corresponding nodes and weights of the quadrature rule in Theorem \ref{Thm-QuadRule-RealLine} are given explicitly. 
In Section \ref{Sec-NumEval-NodesWeights} we explore two numerical techniques for the generation of the nodes and weights of the quadrature rules given by Theorem \ref{Thm-QuadRule-RealLine}. Finally, in Section \ref{Sec-Applications-of-QR} we consider some applications of the quadrature rules given by Theorems \ref{Thm-QuadRule-RealLine} and \ref{Thm-QuadRule-UnitCircle}.

\setcounter{equation}{0}
\section{Some preliminary results  and statements of the main results} \label{Sec-Statements}

There are simple connections between the polynomials $\{P_n\}_{n \geq 0}$ given by \eqref{Eq-TTRR-Special-R2-type} and orthogonal polynomials on the unit circle.  Hence, let us first recall some information regarding measures and orthogonal polynomials on the unit circle.  
 
  Given a sequence of complex numbers $\{\alpha_{n}\}_{n \geq 0}$, where $|\alpha_{n}| < 1$, $n \geq 0$, then it is known that there exists a unique probability measure $\mu$ on the unit circle such that the sequence of monic polynomials $\{\Phi_{n}\}_{n \geq 0}$ generated by
\[
    \Phi_{n+1}(z) = z \Phi_{n}(z) - \overline{\alpha}_{n} \Phi_{n}^{\ast}(z), \quad n \geq 0,
\]
are the monic orthogonal polynomials on the unit circle with respect to the measure $\mu$.  That is,
\[
    \int_{\T} \overline{\Phi_{m}(\zeta)} \Phi_{n}(\zeta) d \mu(\zeta) =  \kappa_{n}^{-2} \delta_{nm}, \quad n,  m = 0, 1, 2, \ldots  , 
\]
where $\T= \{ \zeta = e^{i\theta}:  0\leq \theta < 2\pi\}$. 

 This  unique map between the sequence $\{\alpha_{n}\}_{n \geq 0}$ and the measure $\mu$  is exactly the equivalent result on the unit circle of the so called Favard Theorem on the real line. For a simple proof of this Favard Theorem on the unit circle we cite \cite{ENZG1}. Other proofs can be found in Simon \cite{Simon-book-p1}, where  the sequence  $\{\alpha_{n}\}_{n \geq 0}$, based on \cite{Verblunsky-1935},   is referred to as the sequence of Verblunsky coefficients associated with the measure $\mu$. When necessary, we will use the notation $\alpha_{n}(\mu)$ for the Verblunsky coefficients to indicate their connection to the measure $\mu$. 

Given a probability measure $\mu$ on the unit circle, in order to indicate the size $\epsilon$ of its mass at $\zeta = 1$, we will also use the notation $\mu_{\underline{\epsilon}}$. Hence, the notation $\mu_{\underline{0}}$ means the measure $\mu$ does not have a pure point at $\zeta=1$. 

We now summarize in Theorem \ref{Thm-Basics} below  some results which will be of importance throughout this paper. 

\begin{theorem}  \label{Thm-Basics}
Given the three term recurrence \eqref{Eq-TTRR-Special-R2-type}, consider the map from \linebreak  $(\{c_{n}\}_{n \geq 1},\{d_{n+1}\}_{n \geq 1})$ to  $\{\alpha_n\}_{n \geq 0}$, where $|\alpha_n| < 1$, $n \geq 0$ , given by 
\begin{equation} \label{Eq-Verblunsky-Characterization-1}
    \alpha_{n-1} = - \frac{1}{\tau_{n}}\,\frac{1-2\ell_{n+1} - i c_{n+1}}{1 - i c_{n+1}},  \quad n \geq 1.
\end{equation} 
Here,  $\tau_{n} = \prod_{k=1}^{n} [(1-ic_k)/(1+ic_k)]$, $n \geq 1$ and $\{\ell_{n+1}\}_{n \geq 0}$ is the minimal parameter sequence of the positive chain sequence $\{d_{n+1}\}_{n \geq 1}$. Let $\mu$ be the  probability measure on the unit circle for which $\{\alpha_n\}_{n \geq 0}$ is the sequence of Verblunsky coefficients.   

If   
\[
   1 + \sum_{n = 2}^{\infty} \prod_{k=2}^{n} \frac{\ell_{k}}{1-\ell_{k} }  = \mathpzc{S} < \infty, 
\]
then $\mu$ is such that the integral $ \int_{\T} |\zeta-1|^{-2} d \mu(\zeta)$ exists and takes the value $\mathpzc{A}(\mu) = (c_1^2 + 1)\mathpzc{S}/4$.  In this case, let $\nu_{\underline{0}}$ be the probability measure on the unit circle given by  
\begin{equation}\label{Eq-nu0-measure}
  \nu_{\underline{0}}(e^{i\theta}) = \frac{1}{\mathpzc{A}(\mu)} \int_{0}^{\theta}\frac{1}{|e^{i\Theta}-1|^2}  d \mu(e^{i\Theta}),  
\end{equation}
and let the bounded  non-decreasing function $\varphi$ in $(-\infty, \infty)$ be such that $d \varphi(x)  = \linebreak - d \nu_{\underline{0}}\big((x+i)/(x-i)\big)$.   
Then, for $n \geq 1$, 
\begin{equation} \label{Eq-Orthogonality-for-Pn}
  \begin{array}{c}
     \dsp \int_{-\infty}^{\infty} x^{j} \, \frac{P_{n}(x)}{(x^2 +1)^{n}}\,d\varphi(x) = 0, \quad j = 0,1,2, \ldots, n-1.
  \end{array}
\end{equation}
Moreover,  if 
\[
     \gamma_{n}  =  \int_{-\infty}^{\infty}   \frac{x^{n}P_{n}(x)}{(x^2 +1)^{n}}\,d\varphi(x), \ \ \widehat{\gamma}_{n}=  \int_{-\infty}^{\infty} \frac{(x+i)P_{n}(x)}{(x^2+1)^{n+1}} d \varphi(x), 
\]
for $n \geq 0$,  then $\gamma_{0} = \int_{-\infty}^{\infty} d \varphi(x) =   \int_{\T}  d \nu_{\underline{0}}(\zeta)  = 1$, $1-M_{n} = \gamma_{n}/\gamma_{n-1}$, $n \geq 1$, 
\begin{equation} \label{Eq-Coeff-Rep}
    c_1 = \frac{\int_{-\infty}^{\infty}x(x^2+1)^{-1} d\varphi(x)}{\int_{-\infty}^{\infty}(x^2+1)^{-1} d\varphi(x)} \ \ \mbox{and} \ \   d_{n+1} = \frac{1}{\Im(\widehat{\gamma}_{n-1}/\widehat{\gamma}_{n})}, \ \  c_{n+1} = - \frac{\Re(\widehat{\gamma}_{n-1}/\widehat{\gamma}_{n})}{\Im(\widehat{\gamma}_{n-1}/\widehat{\gamma}_{n})}, 
\end{equation}
for $n \geq 1$. Here, $\{M_{n+1}\}_{n\geq 0}$ is the maximal parameter sequence of $\{d_{n+1}\}_{n \geq 1}$. 

\end{theorem}  

\begin{proof}  Most of the results in this theorem of preliminary results follow from  \cite{Ismail-Ranga} and references therein. The value $\mathpzc{S}$ of the infinite series in Theorem \ref{Thm-Basics} is finite also equivalent of saying, by Wall's criteria,  that the positive chain sequence $\{d_{n+1}\}_{n\geq 1}$ in \eqref{Eq-TTRR-Special-R2-type} has multiple parameter sequences.  In this case, the affirmation  $\mathpzc{A}(\mu) = \int_{\T} \frac{1}{|\zeta-1|^2} d \mu(\zeta) = (c_1^2 + 1)\mathpzc{S}/4$ follows from results given within the proof of  \cite[Thm.\, 3.2]{Ismail-Ranga}. Moreover, we can also identify (see Theorem 6.1 in \cite[p.101]{Chihara-Book}) that  
\begin{equation*} \label{Eq-MaxPara-to-S}
    \frac{1}{M_1} = \mathpzc{S} = \frac{4\mathpzc{A}(\mu)}{c_1^2 + 1},
\end{equation*} 
where $M_{1}$ is the initial element of the maximal parameter sequence of the positive chain sequence $\{d_{n+1}\}_{n\geq 1}$. 

The formulas for  $\{c_{n}\}$ and $\{d_{n+1}\}$ given by \eqref{Eq-Coeff-Rep}, which can be considered as further new results,  are obtained from the real and imaginary parts of 
\[
      (c_{n+1}-i) \int_{-\infty}^{\infty} (x+i) \frac{P_{n}(x)}{(x^2+1)^{n+1}} d\varphi(x) = - d_{n+1} \int_{-\infty}^{\infty} (x+i) \frac{P_{n-1}(x)}{(x^2+1)^{n}} d\varphi(x), \quad  n \ge 1,
\]
The above result follows from multiplication of \eqref{Eq-TTRR-Special-R2-type} by $ (x-i)^{-1}(x^2+1)^{-n}$ and then use of the orthogonality given by \eqref{Eq-Orthogonality-for-Pn}. 
\end{proof}

Observe that the formulas for $\{c_{n}\}$ and $\{d_{n+1}\}$ given by \eqref{Eq-Coeff-Rep}, like the Stieltjes procedure for generating orthogonal polynomials on the real line (see \cite{Gautschi-SiamJSSC82}), provide a recursive method for the numerical construction of the polynomials $P_{n}$ starting from $\varphi$.

\begin{remark} 
 Clearly, it is important that given a probability measure $\mu$ such that the integral  $ \int_{\T} \frac{1}{|\zeta-1|^2} d \mu(\zeta)$ exists then how one could  get from  $\{\alpha_{n}(\mu)\}_{n \geq 0}$ the sequences $\{c_{n}\}_{n \geq 1}$ and $\{d_{n+1}\}_{n \geq 1}$ which, by \eqref{Eq-Verblunsky-Characterization-1}, lead back to  $\{\alpha_{n}(\mu)\}_{n \geq 0}$.  In other words, what is the inverse map of \eqref{Eq-Verblunsky-Characterization-1}? This we give in Theorem \ref{Thm-Inv-Map-1}. 

\end{remark}

If we know the maximal parameter sequence of the positive chain sequence $\{d_{n+1}\}_{n \geq 1}$ in Theorem \ref{Thm-Basics}, then, as in \cite[Thm.~1.2]{Ismail-Ranga}, the Verblunsky coefficients associated with the measure $\nu_{\underline{0}}$ in Theorem \ref{Thm-Basics} are 
\begin{equation} \label{Eq-Verblunsky-Characterization-22}
      \alpha_{n-1}(\nu_{\underline{0}}) = \frac{1}{\tau_{n-1}}\frac{1-2M_n-ic_n}{1-ic_n}, \quad n\geq 1.
\end{equation}
With the probability measure $\nu_{\underline{0}}$  we can generate a family of probability measures $\nu_{\underline{\epsilon}}$, $0 \leq \epsilon < 1$, by the Uvarov transformation 
\begin{equation} \label{Eq-nu-Class}
      \int_{\T} \phi(\zeta)\,d\nu_{\underline{\epsilon}}(\zeta) = (1-\epsilon)\int_{\T} \phi(\zeta)\,d\nu_{\underline{0}}(\zeta) + \epsilon \phi(1).
\end{equation}
Clearly, with  the quantity $\mathpzc{B}(\nu)$ associated with any $\nu$, defined by   
\[ 
     \mathpzc{B}(\nu) = \int_{\T} |\zeta-1|^2 d \nu(\zeta), 
\]
one can easily see that there hold 
\begin{equation*} \label{Eq-nu-to-mu-temp}
     \mu(e^{i\theta}) = \frac{1}{\mathpzc{B}(\nu_{\underline{\epsilon}})}\int_{0}^{\theta} |e^{i\Theta}-1|^2\, d\nu_{\underline{\epsilon}}(e^{i\Theta}) = \frac{1}{\mathpzc{B}(\nu_{\underline{0}})}\int_{0}^{\theta} |e^{i\Theta}-1|^2\, d\nu_{\underline{0}} (e^{i\Theta}), \quad 0 < \epsilon < 1.  
\end{equation*}

\begin{remark}
 Again, it is important that, given a probability measure $\nu = \nu_{\underline{\epsilon}}$, $0 \leq \epsilon <1$,  as above then how one can get directly from the Verblunsky coefficients $\{\alpha_{n}(\nu)\}_{n \geq 0}$ the sequences $\{c_{n}\}_{n \geq 1}$ and $\{d_{n+1}\}_{n \geq 1}$ which by \eqref{Eq-Verblunsky-Characterization-1} lead to  $\{\alpha_{n}(\mu)\}_{n \geq 0}$. That is, what is the inverse map of \eqref{Eq-Verblunsky-Characterization-22}? This we give in Theorem \ref{Thm-Inv-Map-2}.

\end{remark}

The main results of the present paper are the following theorems.

\begin{theorem} \label{Thm-QuadRule-RealLine} Let $\varphi$ be given as in  Theorem \ref{Thm-Basics}.   Let $x_{n,k}$, $k=1,2, \ldots, n$, be the zeros of the polynomial $P_n$ (in decreasing order) and let the numbers $\omega_{n,k}$ be such that 
\begin{equation}\label{Eq1-omega-nk}  
  \omega_{n,k} = \frac{(x_{n,k}^2+1)^{n-1}d_2d_3\cdots d_n M_1}{P'_n(x_{n,k})P_{n-1}(x_{n,k})} = \frac{M_1}{ \mathbf{u}_{n}(x_{n,k})^{H} \mathbf{B}_{n} \mathbf{u}_{n}(x_{n,k})}, \quad k = 1,2, \ldots, n. 
\end{equation} 
Then $\omega_{n,k}$ are all positive and for any $f$ such that $(x^2+1)^nf(x)\in \mathbb{P}_{2n-1}$ there holds the  quadrature rule 
\begin{equation}  \label{Eq0-QR-RL}
\int_{-\infty}^{+\infty}f(x)\,d\varphi(x) = \sum_{k=1}^{n}\omega_{n,k}\,f(x_{n,k}). 
\end{equation} 

\end{theorem}

The proof of  Theorem \ref{Thm-QuadRule-RealLine} is given in Section \ref{Sec-QuadRules-RL}. In this section we also look at alternative ways of representing the weights $\omega_{n,k}$. Polynomials given by 
\begin{equation*} \label{Eq-Qn0-Polynomials}
   Q_{n}(x)=\int_{-\infty}^{+\infty}\frac{(x^2+1)^{n} P_n(t)-(t^2+1)^{n} P_n(x)}{(t-x)}\frac{1}{(t^2+1)^n}d\varphi(t),  \quad n \geq 1,
\end{equation*} 
also play an important role in this section.

\begin{theorem}  \label{Thm-QuadRule-UnitCircle}

Let $\mu$ be the probability measures on the unit circle given by Theorem \ref{Thm-Basics}, obtained under the condition  $\mathpzc{S} < \infty$, and let $x_{n,k}$ and $\omega_{n,k}$, $k = 1,2, , \ldots,n$ be as in Theorem \ref{Thm-QuadRule-RealLine}.   
Then for any $\mathcal{F}(z) \in span\{z^{-n+1},z^{-n+2},\ldots,z^{n-2},z^{n-1}\}$, there holds the $n$-point quadrature rule 
\begin{equation} \label{Eq-QR-UC-mu}
      \int_{\T} \mathcal{F}(\zeta)\, d \mu(\zeta) = \sum_{k=1}^{n} \lambda_{n,k} \, \mathcal{F}(\xi_{n,k}),
\end{equation}
where 
\[
      \xi_{n,k} = \frac{x_{n,k}+i}{x_{n,k}-i} \quad \mbox{and} \quad %
\lambda_{n,k} = \frac{c_1^2 +1}{M_1} \, \frac{\omega_{n,k}}{x_{n,k}^2 + 1} = \mathpzc{A}(\mu) |\xi_{n,k} -1|^2\omega_{n,k}, 
\]
for $ k =1, 2, \ldots, n$. 

For any $\epsilon$ such that $0 \leq \epsilon < 1$, if $\nu_{\underline{\epsilon}}$ is the probability measure given by Theorem \ref{Thm-Basics} and \eqref{Eq-nu-Class}, then for any $\mathcal{F}(z) \in span\{z^{-n},z^{-n+1},\ldots,z^{n-1},z^{n}\}$, there holds the $(n+1)$-point quadrature rule 
\begin{equation} \label{Eq-QR-UC-nu}
      \int_{\T} \mathcal{F}(\zeta)\, d \nu_{\underline{\epsilon}}(\zeta) = [(1-\epsilon) \widehat{\lambda}_{n+1,n+1} + \epsilon]\, \mathcal{F}(1) + \sum_{k=1}^{n} (1-\epsilon)\omega_{n,k} \, \mathcal{F}(\xi_{n,k}),
\end{equation}
where 
\[
       \widehat{\lambda}_{n+1,n+1} = \frac{(1-M_1)(1-M_2) \cdots (1-M_n)}{(1-\ell_1)(1-\ell_2)\cdots(1-\ell_n)}. 
\]
\end{theorem}

  The proof of this theorem is in Section \ref{Sec-QuadRules-UC}. 

\begin{remark} 
{ 
The quadrature rules given by \eqref{Eq-QR-UC-mu}  and \eqref{Eq-QR-UC-nu} are  particular cases of the quadrature rules based on para-orthogonal polynomials introduced by Jones, Nj\aa stad and Thron \cite{JONES}. Precisely which para-orthogonal polynomials lead to these quadrature rules  and, consequently, also what is the inverse maps from  $\{\alpha_n\}_{n \geq 0}$ to   $(\{c_{n}\}_{n \geq 1}, \{\ell_{n+1}\}_{n \geq 1})$, are clarified in Section \ref{Sec-QuadRules-UC}. 
 }
\end{remark}

\setcounter{equation}{0}
\section{Quadrature rules on the real line } 
\label{Sec-QuadRules-RL}

In order to derive the quadrature rule given by  Theorem \ref{Thm-QuadRule-RealLine}, let the rational function $f$ be such that  $(x^2+1)^n f(x) \in \PP_{2n-1}$. Hence, one can write 
\[
    (x^2+1)^n f(x) = q(x)\, P_{n}(x)  + r(x),
\]
with $q(x) \in \PP_{n-1}$, $r(x) \in \PP_{n-1}$ and $r(x_{n,k})= (x_{n,k}^2+1)^{n}f(x_{n,k})$, $k=1,2, \ldots, n$. Here,  $x_{n,k}$, $k=1,2, \ldots, n$, are the $n$ real and simple zeros of $P_{n}$. Writing $r(x)$ in terms of its  interpolatory polynomial at the zeros of $P_{n}$ gives 
\[
    (x^2+1)^n f(x) = q(x) \,P_{n}(x) + \sum_{k=1}^{n} \frac{P_n(x)}{(x-x_{n,k}) P_{n}^{\prime}(x_{n,k})}(x_{n,k}^2+1)^{n}f(x_{n,k}). 
\]
 Thus, from the orthogonal property of $P_{n}$ in Theorem \ref{Thm-Basics}, 
\begin{equation}  \label{Eq1-QR-RL}
    \int_{-\infty}^{\infty} f(x) d \varphi(x) = \sum_{k=1}^{n} \omega_{n,k} \,f(x_{n,k}),
\end{equation}
if one assumes 
\begin{equation} \label{Eq2-omega-nk}
        \omega_{n,k}  = \int_{-\infty}^{+\infty}\frac{P_n(x)}{(x-x_{n,k})P_{n}^{\prime}(x_{n,k})}  
       \frac{(x^{2}_{n,k}+1)^n}{(x^2+1)^n}d\varphi(x), \quad k = 1,2, \ldots, n. 
\end{equation}
This is exactly the quadrature rule given by Theorem \ref{Thm-QuadRule-RealLine}. In order to  confirm this we now look for other representations for the quantities  $\omega_{n,k}$  including the representations given within Theorem \ref{Thm-QuadRule-RealLine}. 

For  $f(x) = (x^2+1)^{-n} \Big[\frac{P_{n}(x)}{(x-x_{n,j})P_{n}^{\prime}(x_{n,j})} \Big]^2 $ there holds $(x^2+1)^n f(x)  \in \PP_{2n-2}$.  Hence, from  \eqref{Eq1-QR-RL},    
\begin{equation*} \label{Eq3-omega-nk}
        \omega_{n,j}  = \int_{-\infty}^{+\infty}\Big[\frac{P_{n}(x)}{(x-x_{n,j})P_{n}^{\prime}(x_{n,j})} \Big]^2  
       \frac{(x_{n,j}^2+1)^n}{(x^2+1)^n}d\varphi(x), \quad j = 1,2, \ldots, n. 
\end{equation*}
This shows that the numbers $\omega_{n,j}$ are all positive. 

Now to obtain the expression for $\omega_{n,k}$ given as in Theorem \ref{Thm-QuadRule-RealLine}, let us  consider the polynomials  $Q_{n}^{(r)}$  defined by
\begin{equation} \label{Eq-Qnr-Polynomials}
   Q_{n}^{(r)}(x)=\int_{-\infty}^{+\infty}\frac{(x^2+1)^{n} t^{r}P_n(t)-(t^2+1)^{n}x^r P_n(x)}{(t-x)(t^2+1)^n} d\varphi(t),  
\end{equation} 
for $r =0, 1, \ldots, n$ and for $n \geq 1$. At a first glance all one can say is that  $Q_{n}^{(r)}$ is  a polynomial of degree $2n-1$ or less. It turns out  $Q_{n}^{(r)}$ is a polynomial  of exact degree $n+r-1$. To be more precise,  we can state the following theorem. 

\begin{proposition} \label{Thm-Qnr-Properties} The polynomial $Q_{n}^{(r)}$, defined as in  \eqref{Eq-Qnr-Polynomials}, is exactly of degree $n +r-1$ and satisfy 
\[
         Q_{n}^{(r)}(x) = x^r Q_{n}^{(0)}(x), \quad r=0,1, \ldots, n, \quad n \geq 1. 
\] 
Moreover, the sequence of polynomials  $\{Q_{n}^{(0)}\}_{n \geq 0} = \{Q_{n}\}_{n \geq 0}$ satisfy
\begin{equation} \label{Eq-TTRR-Special-R2-type-for-Qn}
         Q_{n+1}(x)  = (x-c_{n+1}) Q_{n}(x)  - d_{n+1} (x^2+1) Q_{n-1}(x), \quad n \geq 1,
\end{equation}  
with $Q_{0}(x) = 0$ and $Q_{1}(x) = M_1$. 
\end{proposition} 

\begin{proof}

From \eqref{Eq-Qnr-Polynomials} we have for $r =1,2, \ldots, n$, 
\[ 
\begin{array}{rl}
    Q_{n}^{(r)}(x) = &\dsp  \int_{-\infty}^{+\infty}\frac{(x^2+1)^{n}t^{r}P_n(t)-(x^2+1)^{n} x^r P_n(t)}{(t-x)(t^2+1)^n}d\varphi(t) \\[3ex]
      &\dsp   \quad + \int_{-\infty}^{+\infty}\frac{(x^2+1)^n x^{r}P_n(t)-(t^2+1)^nx^rP_n(x)}{(t-x)(t^2+1)^n}d\varphi(t).
\end{array} 
\]
Clearly the second integral is $x^r Q_{n}^{(0)}(x)$ and, since $(t^r-x^r)/(t-x)$ is a polynomial of degree $r-1$,  from the orthogonal property of $P_{n}$ in Theorem \ref{Thm-Basics} the first integral is identically zero.  Thus, if $Q_{n}^{(0)}(x)$ is of exact degree $n-1$ the first part of Proposition \ref{Thm-Qnr-Properties} is confirmed. This can be verified (see Remark \ref{Rm-ttrr} below) if we can prove the three term recurrence relation satisfied by $\{Q_{n}^{(0)}\}_{n \geq 0}$.

From the $R_{II}$ recurrence \eqref{Eq-TTRR-Special-R2-type} for $\{P_n\}_{n \geq 0}$ we have 
\[ 
  \begin{array}{l}
   \dsp \frac{(x^2+1)^{n}P_{n+1}(t) - (t^2+1)^{n} P_{n+1}(x)}{(t-x)(t^2+1)^{n}} \\[3ex]
   \dsp \hspace{8ex} = \frac{(x^2+1)^{n}(t-c_{n+1})P_{n}(t) - (t^2+1)^{n} (x-c_{n+1})P_{n}(x)}{(t-x)(t^2+1)^{n}} \\[3ex]
   \dsp \hspace{17ex} - \ d_{n+1} \frac{(x^2+1)^{n}P_{n-1}(t) - (t^2+1)^{n-1}(x^2+1) P_{n-1}(x)}{(t-x)(t^2+1)^{n-1}}, 
  \end{array}  
\]
for $n \geq 1$. Thus, integration with respect to $\varphi$ and using the first part Proposition \ref{Thm-Qnr-Properties} gives 
\begin{equation} \label{Eq-TTRR-tempFormula} 
  \begin{array}{l}
   \dsp \int_{-\infty}^{\infty}\frac{(x^2+1)^{n}P_{n+1}(t) - (t^2+1)^{n} P_{n+1}(x)}{(t-x)(t^2+1)^{n}} d \varphi(t)\\[3ex]
   \dsp \hspace{12ex} = (x-c_{n+1})Q_{n}(x) - \ d_{n+1} (x^2+1) Q_{n-1}(x),
  \end{array}  
\end{equation}
for $ n \geq 1$, where $Q_{0}(x) = 0$. Hence, all one needs to verify is that the left hand side of \eqref{Eq-TTRR-tempFormula} represents $Q_{n+1}$ for $n \geq 1$. 

Clearly 
\[
  \begin{array}{l}
    \dsp \int_{-\infty}^{\infty}\frac{(x^2+1)^{n}P_{n+1}(t) - (t^2+1)^{n} P_{n+1}(x)}{(t-x)(t^2+1)^{n}} d \varphi(t) \\[3ex]
    \dsp \hspace{2ex} = \frac{1}{(x^2+1)}  \int_{-\infty}^{\infty}\frac{(x^2+1)^{n+1}(t^2+1)P_{n+1}(t) - (t^2+1)^{n+1} (x^2+1) P_{n+1}(x)}{(t-x)(t^2+1)^{n+1}} d \varphi(t).
  \end{array} 
\]
Hence, from the first part of Proposition \ref{Thm-Qnr-Properties} we conclude that the left hand side of \eqref{Eq-TTRR-tempFormula} is actually $Q_{n+1}(x)$ when $n \geq 1$. Now to complete the proof of Proposition \ref{Thm-Qnr-Properties} all one needs to do is to establish the value of $Q_1$. From $P_1(x) = x- c_1$  and from \eqref{Eq-Qnr-Polynomials}, 
\[
   \begin{array}{rl}
    Q_{1}(x) =&\dsp \int_{-\infty}^{+\infty} \frac{(x^2+1)(t-c_1)-(t^2+1)(x-c_1)}{(t-x)(t^2+1)}d\varphi(t) \\[3ex]
    =&\dsp \int_{-\infty}^{+\infty} \frac{c_1t+1}{t^2+1}d\varphi(t) - x \int_{-\infty}^{+\infty} \frac{t- c_1}{t^2+1}d\varphi(t).
   \end{array}
\] 
Since $P_1(t) = t- c_1$ from the orthogonality of $P_1$ the second integral is zero and thus, 
\[
   \begin{array}{rl}
     Q_{1}(x) = &\dsp  \int_{-\infty}^{+\infty} \frac{c_1t+1}{t^2+1}d\varphi(t) = \int_{-\infty}^{+\infty} \frac{t^2+1}{t^2+1}d\varphi(t) - \int_{-\infty}^{+\infty} \frac{tP_{1}(t)}{t^2+1}d\varphi(t), \\[3ex]
   \end{array}   
\]
where, from Theorem \ref{Thm-Basics},  the first integral on the right hand side is equal to $1$ and the other integral is equal to $1-M_1$. Thus, proving $Q_{1}(x) = M_1$.  This completes the proof. 
\end{proof}

\begin{remark} \label{Rm-ttrr}
Since $\{d_{n+1}\}_{n \geq 1}$ is  a positive chain sequence it follows that $\{d_{n+2}\}_{n \geq 1}$ is also a positive chain sequence (see \cite{Chihara-Book}). Thus, using the three term recurrence for $\{Q_{n}\}_{n \geq 0}$ one can show that $Q_{n}$ is of exact degree $n-1$. Precisely, if $\{\mathfrak{l}_{n+2}\}_{n \geq 0}$ is the minimal parameter sequence of $\{d_{n+2}\}_{n \geq 1}$ then $Q_{n}(x) = M_1 \prod_{j = 0}^{n-1}(1-\mathfrak{l}_{j+2})\, x^{n-1} + \mbox{lower order terms}$. 
\end{remark} 

Now from the three term recurrence relations for $P_{n}$ and $Q_{n}$ one easily finds 
\begin{equation} \label{Eq-Determinant-form} 
  \begin{array}l
   Q_{1}(x) P_{0}(x) - Q_{0}(x) P_{1}(x) = M_1 \quad \mbox{and} \\[1ex] 
   Q_{n}(x) P_{n-1}(x) - Q_{n-1}(x) P_{n}(x) = M_1 d_2 d_3 \cdots d_{n} (x^2+1)^{n-1}, \quad n \geq 2.
  \end{array}
\end{equation}
Hence, $Q_n$ and $P_n$ do not have any common zeros. Moreover, from \eqref{Eq2-omega-nk} and  \eqref{Eq-Qnr-Polynomials}, 
\begin{equation} \label{Eq4-omega-nk}
    \omega_{n,k}  = \frac{Q_{n}(x_{n,k})}{P_{n}^{\prime}(x_{n,k})}, \quad k =1,2, \ldots, n.
\end{equation}

Now the first expression of \eqref{Eq1-omega-nk} for $\omega_{n,k}$ in Theorem \ref{Thm-QuadRule-RealLine} is an immediate  consequence of \eqref{Eq-Determinant-form} and \eqref{Eq4-omega-nk}. Now to obtain the matrix expression for $\omega_{n,k}$ we use \eqref{Eq-PosDef-formula}.  Thus, concluding the proof of Theorem \ref{Thm-QuadRule-RealLine}.  \hfill  \qed  \\[-1ex]

For the polynomials $Q_{n}(x) = Q_{n}^{(0)}(x)$ the expression  \eqref{Eq-Qnr-Polynomials} can also be written as 
\[
  \begin{array}{cl}
    Q_{n}(x) = &\dsp  (x-i)^n\int_{-\infty}^{+\infty}\frac{[(x+i)^n-(t+i)^n]P_n(t)}{(t-x)(t^2+1)^n}d\varphi(t) \\[3ex]
       &\dsp  \qquad \quad + \ \int_{-\infty}^{+\infty}\frac{[(x-i)^nP_n(t)-(t-i)^nP_n(x)](t+i)^n}{(t-x)(t^2+1)^n}d\varphi(t), 
  \end{array}
\]
for $n \geq 1$. Again from the orthogonal property of $P_n$ the first integral above is zero and thus, we have the following alternative expression for $Q_{n}$:  
\begin{equation} \label{Eq-Qn0-Polynomials-Alternative}
    Q_{n}(x) = \int_{-\infty}^{+\infty}\frac{[(x-i)^{n}P_n(t)-(t-i)^{n}P_n(x)]}{(t-x)}\frac{1}{(t-i)^n}d\varphi(t), \quad n \geq 1. 
\end{equation}
Clearly,  $[(x-i)^n P_n(t)-(t-i)^n P_n(x)]/(t-x)$ is a polynomial of degree $n-1$ in $x$ and hence, this also shows that  $Q_{n}(x)$ is a polynomial of degree not exceeding $n-1$. 

From \eqref{Eq-Qn0-Polynomials-Alternative} we also have as another expression for $\omega_{n,k}$,
\[
    \omega_{n,k}  = \int_{-\infty}^{+\infty}\frac{P_n(x)}{(x-x_{n,k})P_{n}^{\prime}(x_{n,k})}\frac{(x_{n,k}-i)^{n}}{(x-i)^n}d\varphi(x), \quad  k = 1, 2, \ldots, n.
\]

Finally, using \eqref{Eq4-omega-nk} observe that for the rational functions $Q_{n}/P_{n}$, we also have  
\[
    \frac{Q_{n}(x)}{P_{n}(x)} = \sum_{k=1}^{n} \frac{Q_{n}(x_{n,k})/P_{n}^{\prime}(x_{n,k})}{x -x_{n,k}} = \sum_{k=1}^{n} \frac{\omega_{n,k}}{x -x_{n,k}}. 
\]
%

\setcounter{equation}{0}
\section{Quadrature rules on the unit circle} 
\label{Sec-QuadRules-UC}

Interpolatory quadrature rules on the unit circle, first explicitly appeared in \cite{JONES}, are based on the zeros of para-orthogonal polynomials on the unit circle. Given a positive measure $\mu$ on the unit circle,  let us denote by $\{\Phi_{n}(\mu;z)\}_{n \geq 0}$ and  $\{\alpha_{n}(\mu)\}_{n \geq0}$, respectively,  the associated  sequences of monic orthogonal polynomials and  Verblunsky coefficients.   

 For any given $n \geq 1$  let    
\begin{equation} \label{Eq-Notation-POP}
     \Psi_{n}(\mu; \rho, z)  = z \Phi_{n-1}(\mu;z) - \rho \Phi_{n-1}^{\ast}(\mu;z), 
\end{equation}
where $\rho$ is such that $|\rho| = 1$.  The monic polynomial $\Psi_{n}(\mu; \rho, z)$ of degree $n$ in $z$ is known as a para-orthogonal polynomial and its  zeros $\Xi_{n,j}(\mu; \rho)$, $j =1,2, \ldots n$,  are all simple and lie exactly on the unit circle $|z|=1$. Moreover, the  quadrature rule  
\[
     \int_{\T} \mathcal{F}(\zeta) d \mu(\zeta) = \sum_{k=1}^{n} \Lambda_{n,k}(\mu; \rho)\, \mathcal{F}\big(\Xi_{n,k}(\mu; \rho)\big) 
\]
is valid for $\mathcal{F}(z) \in span\{z^{-n+1},z^{-n+2},\ldots,z^{n-2},z^{n-1}\}$ if 
\begin{equation} \label{Eq-Lambda-Interpol-Form}
      \Lambda_{n,k}(\mu; \rho) = \frac{1}{\Psi_{n}^{\prime}\big(\mu; \rho, \Xi_{n,k}(\mu; \rho)\big)} \int_{\T} \frac{\Psi_{n}(\mu; \rho, \zeta)}{\zeta - \Xi_{n,k}(\mu; \rho)} \, d \mu(\zeta), \quad k =1, 2, \ldots, n.
\end{equation}
The above statements regarding para-orthogonal polynomials and associated quadrature rules follow from \cite{JONES}. 

With specific choices of $\{\rho_{n}\}_{n\geq 0}$, such that $|\rho_{n}| =1$, $n \geq 0$, the resulting sequences of para-orthogonal polynomials $\{\Psi_{n}(\mu; \rho_{n-1}, z)\}_{n \geq 1}$ can be made to satisfy nice three term recurrence relations (see, for example, \cite{BSRS_2016} and \cite{Costa-Felix-Ranga-JAT2013}).

From the expression that connects $\xi_{n,k}$ and  $x_{n,k}$ in Theorem \ref{Thm-QuadRule-UnitCircle}, if we consider the sequence of polynomials $\{R_{n}\}_{n \geq 0}$ given by 
\begin{equation} \label{Eq-R-to-P}
    R_{n}(\zeta) = 
        R_{n}\Big(\frac{x+i}{x-i}\Big) = \frac{2^{n}}{(x-i)^{n}} P_{n}(x), \quad n \geq 0, 
\end{equation}
then $\xi_{n,k}$, $k =1,2, \ldots, n$ are the zeros of $R_{n}(z)$ and further, from \eqref{Eq-TTRR-Special-R2-type}, 
\begin{equation} \label{Eq-TTRR-Special-ParaOrth-polynomials}
    R_{n+1}(z) = [(1+ic_{n+1})z+(1-ic_{n+1})] R_{n}(z) - 4 d_{n+1} z R_{n-1}(z), \quad n \geq 1, 
\end{equation}  
with $R_{0}(z) = 1$ and $R_{1}(z) = (1+ic_{1})z+(1-ic_{1})$.

 With such knowledge we can state the following theorem, which gives the exact inverse map of \eqref{Eq-Verblunsky-Characterization-1} in Theorem \ref{Thm-Basics}.  The proof of Theorem \ref{Thm-Inv-Map-1} below also gives information about the para-orthogonal polynomials that correspond to the first quadrature rule given by Theorem \ref{Thm-QuadRule-UnitCircle}. 

\begin{theorem} \label{Thm-Inv-Map-1}

Let $\mu$ be a probability measure on the unit circle such that the integral $\mathpzc{A}(\mu) = \int_{\T} |\zeta-1|^{-2} d \mu(\zeta)$ exists.  Let $I(\mu) = \int_{\T} \zeta(\zeta-1)^{-1} d\mu(\zeta)$ and let $\{\alpha_{n}\}_{n \geq 0} =\{\alpha_{n}(\mu)\}_{n \geq 0}$ be the associated sequence of Verblunsky coefficients.

Then $\mu$ is the probability measures on the unit circle given by Theorem \ref{Thm-Basics} under the condition  $\mathpzc{S} < \infty$, if and only if, 
\[
      \tau_1 = \frac{\,I(\mu)\,}{\overline{I(\mu)}}, \quad   c_{1} = i \frac{\tau_1 - 1}{\tau_1 + 1} = - \frac{2\Im(\tau_1)}{|\tau_1 + 1|^2}, 
\]
\begin{equation} \label{Eq-alpha-to-cn-and-elln} 
     c_{n+1}  = \frac{\Im(\tau_n \alpha_{n-1})}{\Re(1+\tau_n\alpha_{n-1})}, \quad \ell_{n+1} = \frac{1}{2}\frac{|1+\tau_n\alpha_{n-1}|^2}{\Re(1+\tau_n\alpha_{n-1})}, \quad n \geq 1,
\end{equation} 
where 
\begin{equation} \label{Eq-RecRelation-taun}
     \tau_{n+1} = \tau_{n} \, \frac{1 + \overline{\tau_n \alpha_{n-1}}}{1 + \tau_n \alpha_{n-1}}, \quad n \geq 1. 
\end{equation}

\end{theorem} 

\begin{proof}

From \eqref{Eq-Verblunsky-Characterization-1}, 
\[
     1 + \tau_{n} \alpha_{n-1} = \frac{2 \ell_{n+1}}{1-ic_{n+1}} = \frac{2\ell_{n+1}}{1+c_{n+1}^2} (1+i c_{n+1}), \quad n \geq 1.
\]
Hence, by considering the real and imaginary parts, we find  $\{c_{n+1}\}_{n \geq 1}$ and  $\{\ell_{n+1}\}_{n \geq 1}$ are as in \eqref{Eq-alpha-to-cn-and-elln}. 

Moreover, from $\tau_{n+1}/\tau_{n} = (1-ic_{n+1})(1+ic_{n+1})$, the sequence $\{\tau_{n}\}_{n \geq 1}$ can be obtained from $\{\alpha_{n}\}_{n \geq 0}$ and $\tau_1 = (1-ic_1)/(1+ic_1)$ by \eqref{Eq-RecRelation-taun}. Hence, all one needs to establish is how to choose the value of $\tau_1$. 

The choice $\tau_1 = I(\mu)/\overline{I(\mu)}$ follows from \cite[Thm.\,4.2]{BSRS_2016}, by observing that the polynomials $\{R_{n}\}_{n \geq 0}$ given by $R_{0}(z) =1$ and 
\[
    R_{n}(z) \prod_{k=0}^{n-1}\frac{1-\Re(\tau_{k+1}\alpha_{k-1})}{1-\overline{\tau_{k+1}\alpha_{k-1}}} = \Psi_{n}(\mu; -\tau_{n}, z), \quad n \geq 1,
\] 
where one must take $\alpha_{-1} = -1$, satisfy the three term recurrence relation \eqref{Eq-TTRR-Special-ParaOrth-polynomials} and that 
\begin{equation} \label{Eq-Orth-Prop-Rn}
      \int_{\T}\zeta^{-n+k} R_n(\zeta) \frac{\zeta}{\zeta - 1}d \mu(\zeta) = 0, \quad k=0,1, \ldots, n-1.
\end{equation}
From \eqref{Eq-R-to-P} observe that this latter orthogonality is equivalent to \eqref{Eq-Orthogonality-for-Pn}. This completes the proof of Theorem \ref{Thm-Inv-Map-1}. 
\end{proof}

\begin{remark}
{  In \cite[Thm.\,4.2]{BSRS_2016} the results were derived under a weak  assumption that the measure $\mu$ is such that only the principal value  integral $I(\mu) = \dashint_{\T}\zeta(\zeta-1)^{-1} d\mu(\zeta)$ need to exist. However, in the present case we have made a stronger assumption on $\mu$ such that $\int_{\T} |\zeta-1|^{-2} d \mu(\zeta)$ exists and hence,  $\dashint_{\T}\zeta(\zeta-1)^{-1} d\mu(\zeta) = \int_{\T}\zeta(\zeta-1)^{-1} d\mu(\zeta)$ also holds. 
}
\end{remark}

Now let $\nu$ be any probability measure on the unit circle such that 
\begin{equation} \label{Eq-nu-to-mu}
     \mu(e^{i\theta}) = \frac{1}{\mathpzc{B}(\nu)}\int_{0}^{\theta} |e^{i\Theta}-1|^2\, d\nu(e^{i\Theta}),  
\end{equation}
where we recall that  $\mathpzc{B}(\nu) = \int_{\T} |\zeta-1|^2 d\nu(\zeta)$. Let us also assume that $\nu$ be such that it has a pure point of size $\delta$ at $\zeta=1$. That is, we can use the notation $\nu_{\underline{\delta}}$ for $\nu$.

Hence, with the Uvarov transformation as in \eqref{Eq-nu-Class}, we can also generate a family of probability measures $\nu_{\underline{\epsilon}}$ for $0 \leq \epsilon < 1$. Precisely,  
\begin{equation} \label{Eq-nu-Class2}
      \int_{\T} \phi(\zeta)\,d\nu_{\underline{\epsilon}}(\zeta) = \frac{1-\epsilon}{1-\delta}\int_{\T} \phi(\zeta)\,d\nu_{\underline{\delta}} + \frac{\epsilon-\delta}{1-\delta} \phi(1).
\end{equation}
Observe that for any $\epsilon$ such that $0 \leq \epsilon < 1$, 
\begin{equation} \label{Eq-nu-mu-relations}
     \mu(e^{i\theta}) = \frac{1}{\mathpzc{B}(\nu_{\underline{\epsilon}})}\int_{0}^{\theta} |e^{i\Theta}-1|^2\, d\nu_{\underline{\epsilon}}(e^{i\Theta}) \ \  \mbox{and} \ \  \nu_{\underline{0}}(e^{i\theta}) = \frac{1}{\mathpzc{A}(\mu)} \int_{0}^{\theta}\frac{1}{|e^{i\Theta}-1|^2}  d \mu(e^{i\Theta}).
\end{equation}

\begin{theorem} \label{Thm-Inv-Map-2}

Given any probability measure  $\nu$ on the unit circle, let the  probability measures on the unit circle  $\mu$ and $\nu_{\underline{\epsilon}}$ be those  as in  \eqref{Eq-nu-to-mu}, \eqref{Eq-nu-Class2} and \eqref{Eq-nu-mu-relations}. Moreover, let $\{\alpha_{n}(\nu_{\underline{\epsilon}})\}_{n\geq 0}$ be the Verblunsky coefficients associated with the measure $\nu_{\underline{\epsilon}}$.

Then $\mu$ is the probability measures on the unit circle given by Theorem \ref{Thm-Basics}, if and only if, 
\begin{equation*} \label{Eq-alphaNu-to-cn-and-elln} 
  \begin{array}{c}
      d_{n+1} =  \big[1-g_{n}(\nu_{\underline{\epsilon}})\big] g_{n+1}(\nu_{\underline{\epsilon}}),   \\[2ex]
    \dsp  c_{n}  =  \frac{-\Im\big(\tau_{n-1}\alpha_{n-1}(\nu_{\underline{\epsilon}})\big)}{1 -\Re\big(\tau_{n-1}\alpha_{n-1}(\nu_{\underline{\epsilon}})\big)}\quad \mbox{and} \quad g_{n}(\nu_{\underline{\epsilon}}) = \frac{1}{2}\, \frac{\big|1 - \tau_{n-1}\alpha_{n-1}(\nu_{\underline{\epsilon}})\big|^2}{1 -\Re\big(\tau_{n-1}\alpha_{n-1}(\nu_{\underline{\epsilon}})\big)}, 
  \end{array}
\end{equation*} 
for $n \geq 1$, where $\tau_{0} = 1$ and $\{\tau_{n}\}_{n\geq 1}$ is the same as in Theorems \ref{Thm-Basics} and \ref{Thm-Inv-Map-1}, but can also be derived by the following alternative recurrence 
\begin{equation} \label{Eq-TauofNu-RR}
     \tau_{n} =  \tau_{n-1} \, \frac{1 - \overline{\tau_{n-1}\alpha_{n-1}(\nu_{\underline{\epsilon}})}}{1 - \tau_{n-1}\alpha_{n-1}(\nu_{\underline{\epsilon}})}, \ \ n \geq 1. 
\end{equation} 

\end{theorem} 

\begin{proof}

The proof of this theorem follows from results established in \cite{BSRS_2016} and \cite{Costa-Felix-Ranga-JAT2013}. To sketch the direction behind the proof, we consider the  monic para-orthogonal polynomials  
$\Psi_{n+1}(\nu_{\underline{\epsilon}}; \tau_{n}, z)$,  $n \geq 0$, where $\tau_{n}= \Phi_{n}(\nu_{\underline{\epsilon}};1)/\Phi_{n}^{\ast}(\nu_{\underline{\epsilon}};1)$, $n \geq 0$. From the recurrence relation for $\{\Phi_{n}(\nu_{\underline{\epsilon}};z)\}_{n \geq 0}$, one can easily  verify that $\{\tau_{n}\}_{n \geq 0}$ satisfies \eqref{Eq-TauofNu-RR}. The polynomials $\Psi_{n+1}(\nu_{\underline{\epsilon}}; \tau_{n}, z)/(z-1)$ are modified kernel polynomials (or CD-kernels) and that 
\begin{equation} \label{Eq-nu-Psi-R}
     R_{n}(z) = \frac{\prod_{j=0}^{n-1} \big[1- \tau_{j}\alpha_{j}(\nu_{\underline{\epsilon}})\big]}{\prod_{j=0}^{n-1} \big[1-\mathcal{R}e\big(\tau_{j}\alpha_{j}(\nu_{\underline{\epsilon}})\big)\big]} \,\frac{\Psi_{n+1}\big(\nu_{\underline{\epsilon}}; \tau_{n}, z\big)}{z-1}, \quad n \geq 1,
\end{equation}
satisfy the three term recurrence relation \eqref{Eq-TTRR-Special-ParaOrth-polynomials} follows from \cite{Costa-Felix-Ranga-JAT2013}. The sequence $\{R_{n}\}_{n \geq 0}$ satisfying the orthogonality property \eqref{Eq-Orth-Prop-Rn} can also be easily verified. With these observations we establish the proof of Theorem \ref{Thm-Inv-Map-2}.  
\end{proof} 	

\begin{remark} It is important to observe that the values of sequences $\{\tau_n\}_{n \geq 0}$, $\{c_n\}_{n \geq 1}$ and $\{d_{n+1}\}_{n \geq 1}$ remain the same for any $\nu_{\underline{\epsilon}}$ such that $0 \leq \epsilon < 1$. Again, they are same as those in Theorems \ref{Thm-Basics} and \ref{Thm-Inv-Map-1}. Also as shown in \cite{Costa-Felix-Ranga-JAT2013}, the sequence $\{g_{n+1}(\nu_{\underline{\epsilon}})\}_{n \geq 0}$, which varies with $\epsilon$,  is  a parameter sequence of the positive chain sequence $\{d_{n+1}\}_{n \geq 1}$. The sequence $\{g_{n+1}(\nu_{\underline{0}})\}_{n \geq 0} = \{M_{n+1}\}_{n \geq 0}$ is the maximal parameter sequence of $\{d_{n+1}\}_{n \geq 1}$. 

\end{remark}

Now we can consider the proof of Theorem \ref{Thm-QuadRule-UnitCircle}. 

\begin{proof}[Proof of Theorem \ref{Thm-QuadRule-UnitCircle}]

From \eqref{Eq-R-to-P},  if $x_{n,k}$, $k=1,2, \ldots, n$ are the zeros of $P_{n}$ then $\xi_{n,k} = (x_{n,k} +i )/(x_{n,k} - i)$, $k =1, 2, \ldots, n$ are the zeros of $R_{n}(z)$, or equivalently, the zeros of the monic para-orthogonal polynomial $\Psi_{n}\big(\mu; -\tau_n, z\big)$.

Assuming $\mathpzc{A}(\mu) < \infty$,  we consider the $n$ point interpolatory quadrature rule 
\[
     \int_{\T} \mathcal{F}(\zeta) d \mu(\zeta) = \sum_{k=1}^{n} \lambda_{n,k}\, \mathcal{F}\big(\xi_{n,k}\big), 
\]
on the zeros $\xi_{n,k} = \Xi_{n,k}(\mu;-\tau_n)$ of $\Psi_{n}\big(\mu; -\tau_{n}, z\big)$, which is valid for any 
\[  
     \mathcal{F}(z) \in span\{z^{-n+1},z^{-n+2},\ldots,z^{n-2},z^{n-1}\}. 
\]
Since, 
\[
    \frac{(z+1)^{r} (z-1)^{2n-2-r}}{z^{n-1}}, \quad r=0,1 \ldots, 2n-2,
\]
is a basis for $span\{z^{-n+1},z^{-n+2},\ldots,z^{n-2},z^{n-1}\}$, the coefficients $\lambda_{n,k}= \Lambda_{n,k}(\mu; -\tau_n)$ should be  uniquely determined by 
\begin{equation} \label{Eq-Special-UC-QuadratureEquality}
     \int_{\T} \frac{(\zeta+1)^{r} (\zeta-1)^{2n-2-r}}{\zeta^{n-1}} d \mu(\zeta) = \sum_{k=1}^{n} \lambda_{n,k}\, \frac{(\xi_{n,k}+1)^{r} (\xi_{n,k}-1)^{2n-2-r}}{\xi_{n,k}^{n-1}}, 
\end{equation}
for $r = 0, 1, \ldots, 2n-2$. 

Now we consider the quadrature rule given by Theorem \ref{Thm-QuadRule-RealLine}  which holds for any $f$ such that $(x^2+1)^n f(x) \in \PP_{2n-1}$.  Thus, we have 
\begin{equation} \label{Eq-QR-with-RealCanonicBasis}
     \int_{-\infty}^{\infty} \frac{x^r}{(x^2+1)^{n}} d \varphi(x) = \sum_{k=1}^{n} \omega_{n,k} \frac{x_{n,k}^r}{(x_{n,k}^2+1)^{n}}, \quad r=0,1, \ldots, 2n-1. 
\end{equation}
Hence, by using in \eqref{Eq-QR-with-RealCanonicBasis} the transformation  $\zeta = (x+i)/(x-i)$ together with the results given by Theorem \ref{Thm-Basics}, but only for $r =0,1, \ldots, 2n-2$,  one finds 
\[ 
   \begin{array}{l}
     \dsp \int_{\T} \frac{(\zeta+1)^{r} (\zeta-1)^{2n-2-r}}{\zeta^{n-1}} d \mu(\zeta) \\[3ex]
     \dsp \hspace{16ex} = \   \sum_{k=1}^{n} \mathpzc{A}(\mu)  \omega_{n,k}\frac{(\xi_{n,k}-1)^2}{-\xi_{n,k}}\, \frac{(\xi_{n,k}+1)^{r} (\xi_{n,k}-1)^{2n-2-r}}{\xi_{n,k}^{n-1}}, 
  \end{array}
\]
for $r = 0, 1, \ldots, 2n-2$. Thus, comparing this with \eqref{Eq-Special-UC-QuadratureEquality} we get 
\[
     \lambda_{n,k} = \mathpzc{A}(\mu)  \omega_{n,k}\frac{(\xi_{n,k}-1)^2}{-\xi_{n,k}} = \mathpzc{A}(\mu)  \omega_{n,k} |\xi_{n,k}-1|^2, \quad k =1,2, \ldots, n, 
\] 
and the results associated with  the quadrature rule \eqref{Eq-QR-UC-mu} in Theorem \ref{Thm-QuadRule-UnitCircle} are confirmed.\bigskip

We now consider the interpolatory quadrature rule 
\[
     \int_{\T} \phi(\zeta) d \nu_{\underline{0}}(\zeta) = \widehat{\lambda}_{n+1, n+1}\, \phi(1)\ + \  \sum_{k=1}^{n} \widehat{\lambda}_{n+1,k}\, \phi\big(\xi_{n,k}\big), 
\]
based on the zeros of the para-orthogonal polynomials $\Psi_{n+1}\big(\nu_{\underline{0}}; \tau_n, z\big)$, which holds for 
any $\phi$ in  $span\{z^{-n},z^{-n+1},\ldots,z^{n-1},z^{n}\}$. 

Since
\[
    \frac{(z+1)^{r} (z-1)^{2n-r}}{z^{n}}, \quad r=0,1 \ldots, 2n,
\]
is a basis for $span\{z^{-n},z^{-n+1},\ldots,z^{n-1},z^{n}\}$, the coefficients $\widehat{\lambda}_{n+1,k}$ should be  uniquely determined from  
\[
    \int_{\T} \frac{(\zeta+1)^{2n}}{\zeta^{n}} d \nu_{\underline{0}}(\zeta) = 2^{2n} \widehat{\lambda}_{n+1, n+1} \ + \  \sum_{k=1}^{n} \widehat{\lambda}_{n+1,k}\, \frac{(\xi_{n,k}+1)^{2n} }{\xi_{n,k}^{n}} 
\]
and
\begin{equation} \label{Eq-QR2-basis}
    \int_{\T} \frac{(\zeta+1)^{r}(\zeta-1)^{2n-r}}{\zeta^{n}} d \nu_{\underline{0}}(\zeta) =   \sum_{k=1}^{n} \widehat{\lambda}_{n+1,k}\, \frac{(\xi_{n,k}+1)^{r} (\xi_{n,k}-1)^{2n-r}}{\xi_{n,k}^{n}}, 
\end{equation} 
for $r = 0, 1, \ldots, 2n-1$.  

However, from \eqref{Eq-nu-mu-relations} and \eqref{Eq-QR-with-RealCanonicBasis}, by using the transformation  $\zeta = (x+i)/(x-i)$ we also have 
\begin{equation} \label{Eq-QR-TransFrom-RealCanonicBasis}
     \int_{\T} \frac{(\zeta+1)^{r} (\zeta-1)^{2n-r}}{\zeta^{n}} d \nu_{\underline{0}}(\zeta) = \sum_{k=1}^{n}   \omega_{n,k}\, \frac{(\xi_{n,k}+1)^{r} (\xi_{n,k}-1)^{2n-r}}{\xi_{n,k}^{n}}, 
\end{equation}
for $r = 0, 1, \ldots, 2n-1$. 

Comparing \eqref{Eq-QR-TransFrom-RealCanonicBasis} with \eqref{Eq-QR2-basis} we find  
\[
     \int_{\T} \phi(\zeta) d \nu_{\underline{0}}(\zeta) = \widehat{\lambda}_{n+1, n+1}\, \phi(1)\ + \  \sum_{k=1}^{n} \omega_{n,k}\, \phi\big(\xi_{n,k}\big), 
\]
for $\phi \in span\{z^{-n},z^{-n+1},\ldots,z^{n-1},z^{n}\}$. To obtain the explicit expression for $ \widehat{\lambda}_{n+1, n+1}$ in Theorem \ref{Thm-QuadRule-UnitCircle}, we have from \eqref{Eq-Lambda-Interpol-Form} 
\[
   \begin{array}{rl}
      \widehat{\lambda}_{n+1, n+1} =&  \Lambda_{n+1,n+1}(\nu_{\underline{0}}; \tau_{n})  \\[1ex]
      =& \dsp \frac{1}{\Psi_{n+1}^{\prime}(\nu_{\underline{0}}; \tau_{n}, 1)} \int_{\T} \frac{\Psi_{n+1}(\nu_{\underline{0}}; \tau_n, \zeta)}{\zeta - 1} \, d \nu_{\underline{0}}(\zeta).
    \end{array}
\]
Thus, from  \eqref{Eq-nu-Psi-R},   
\[
   \begin{array}{rl}
      \widehat{\lambda}_{n+1, n+1} 
      =& \dsp \frac{1}{R_{n}(1)} \int_{\T} R_{n}(\zeta)\, d \nu_{\underline{0}}(\zeta).
    \end{array}
\]
However, using  \eqref{Eq-R-to-P} one finds
\[
    R_{n}(1) = \lim_{x \to \infty} \frac{2^{n} P_{n}(x)}{(x-i)^{n}} \ \  \mbox{and} \ \ \int_{\T} R_{n}(\zeta)\, d \nu_{\underline{0}}(\zeta) = \int_{-\infty}^{\infty} \frac{2^{n}(x+i)^{n}P_{n}(x)}{(x^2+1)^{n}} d \varphi(x).
\]
Hence, from \eqref{Eq-LeadCoeff-Pn} and \eqref{Eq-Orthogonality-for-Pn} we obtain 
\[
   \widehat{\lambda}_{n+1, n+1} = \frac{(1-M_1)(1-M_2) \cdots (1-M_n)}{(1-\ell_1)(1-\ell_2)\cdots(1-\ell_n)}. 
\]
Thus, confirming the results corresponding to the  quadrature rule associated with $\nu_{\underline{0}}$ in Theorem \ref{Thm-QuadRule-UnitCircle}.  

Finally,  the results corresponding to the quadrature rule associated with $\nu_{\underline{\epsilon}}$ simply follow from \eqref{Eq-nu-Class}.  Thus, completing the proof of Theorem \ref{Thm-QuadRule-UnitCircle}.   
\end{proof}

\setcounter{equation}{0}
\section{A simple example } 
\label{Sec-Eample-simple}

Consider the polynomial given by
\begin{equation*} \label{Eq-recorrencia-Example}
P_{n+1}(x)=(x-c_{n+1})P_n(x)-d_{n+1}(x^2+1)P_{n-1}(x), \quad n \geq 1,
\end{equation*}
with $P_0(x)=1$, $P_1(x)=x-c_1$, where  $c_{n}=0$ and $d_{n+1}=1/4$,  $n \geq 1$. 

From the theory of difference equation it is easily found that 
\begin{equation} \label{Eq-ExplicitForm-Example}
 P_{n}(x)=i\big(\frac{x-i}{2}\big)^{n+1}-i\big(\frac{x+i}{2}\big)^{n+1}, \quad n \geq 0. 
\end{equation}
Hence,  
\begin{equation}  \label{Eq-ExplicitDerivative-Example}
   P_{n}^{\prime}(x)=\frac{n+1}{2}\big[i\big(\frac{x-i}{2}\big)^{n}-i\big(\frac{x+i}{2}\big)^{n}\big] = \frac{n+1}{2} P_{n-1}(x), \quad n \geq 1. 
\end{equation}
From \eqref{Eq-ExplicitForm-Example} it is also easily verified that the zeros $x_{n,k}$ of $P_{n}$ are such that 
\begin{equation}  \label{Eq-zeros-Rn-Example}
      \xi_{n,k} = \frac{x_{n,k}+i}{x_{n,k}-i} = e^{i2k\pi/(n+1)}, \quad  k =1,2, \ldots, n,
\end{equation}
from which $x_{n,k} = \cot\big(k\pi/(n+1)\big)$, $k =1,2, \ldots, n$.

We first consider the results that correspond to those given by Theorem \ref{Thm-Basics}. The sequence $\{d_{n+1}\}_{n \geq 1} = \{1/4\}_{n \geq 1}$ is known to be positive chain sequence with its minimal $\{\ell_{n+1}\}_{n \geq 0}$ and maximal $\{M_{n+1}\}_{n \geq 0}$ parameter sequences  given by 
\[
    \ell_{n+1} = \frac{n}{2n+2} \ \ \mbox{and} \ \  M_{n+1} = \frac{1}{2}, \quad n \geq 0. 
\]
Thus,
\[
    \mathpzc{S}  =   1 + \sum_{n = 2}^{\infty} \prod_{k=2}^{n} \frac{\ell_{k}}{1-\ell_{k} }  = 1 + \sum_{n = 2}^{\infty} \prod_{k=2}^{n} \frac{k-1}{k+1} = 1 + \sum_{n = 2}^{\infty}  \frac{2}{n(n+1)} = 2,
\]
and the probability measure $\mu$ that follows from the Verblunsky coefficients given by \eqref{Eq-Verblunsky-Characterization-1} is such that $\int_{\T} |\zeta-1|^{-2} d \mu(\zeta)$ exists and takes the value
\[
    \mathpzc{A}(\mu) = \frac{1}{4}(c_1^2 + 1)\mathpzc{S} = \frac{1}{2}.
\]
From \eqref{Eq-Verblunsky-Characterization-1} direct calculations show that 
\[
     \tau_{n} = 1 \quad \mbox{and} \quad \alpha_{n-1} =\alpha_{n-1}(\mu) = -\frac{1}{n+1}, \quad n \geq 1.   
\]
It is known that (see, for example, \cite[Thm.\,8]{BracSilvaRanga-AMC2016}) the associated probability measure $\mu$ is such that 
\[ 
    d\mu(\zeta) = \frac{1}{4\pi i}\frac{(1-\zeta)(\zeta-1)}{\zeta^2}d \zeta, 
\]
and hence, the value of $\mathpzc{A}(\mu)$ is confirmed. Moreover, from this we also find that the probability measure $\nu_{\underline{0}}$ is actually the Lebesgue measure given by $d \nu_{\underline{0}}(\zeta) = \frac{1}{2\pi i}\frac{1}{\zeta} d \zeta$ and further 
\[
   d \varphi(x)  =  - d \nu_{\underline{0}}\big(\frac{x+i}{x-i}\big) = \frac{1}{\pi} \frac{1}{x^2+1} dx. 
\]

Now we consider the results that correspond to those given by Theorem \ref{Thm-QuadRule-RealLine}.  First observe that 
\[
     x_{n,k}^2+1 = \frac{1}{\sin^2\big(k\pi/(n+1)\big)}.
\]
Furthermore, from \eqref{Eq-ExplicitDerivative-Example}, 
\[
   \frac{2}{n+1} P_{n}^{\prime}(x_{n,k}) =  P_{n-1}(x_{n,k}) =  \frac{2\sin\big(nk\pi/(n+1)\big)}{\left[2\sin\big(k\pi/(n+1)\big)\right]^{n}} =  \frac{(-1)^{k-1}}{\left[2\sin\big(k\pi/(n+1)\big)\right]^{n-1}},
\]
for $k = 1,2, \ldots, n$. Hence, from \eqref{Eq1-omega-nk} we have $\omega_{n,k} = 1/(n+1)$ and there follows
\begin{equation}  \label{Eq-QR-RL-Example}
  \frac{1}{\pi} \int_{-\infty}^{+\infty}f(x)\,\frac{1}{x^2+1} dx = \frac{1}{n+1}  \sum_{k=1}^{n} f(x_{n,k}), 
\end{equation} 
which holds whenever $(x^2+1)^nf(x)\in \mathbb{P}_{2n-1}$. 

\begin{remark}
{   The polynomials $Q_{n}(x)=Q_{n}^{(r)}(x)$ defined by \eqref{Eq-Qnr-Polynomials} also play a prominent role in the understanding of the quadrature rule given by Theorem \ref{Thm-QuadRule-RealLine}.  With the values of $\{c_{n}\}_{n \geq 1}$ and $\{d_{n+1}\}_{n \geq 1}$  chosen here we obtain from the three term recurrence relation \eqref{Eq-TTRR-Special-R2-type-for-Qn} that $Q_{n}(x) = M_1 P_{n-1}(x)$.  Thus, $\omega_{n,k} = 1/(n+1)$ also immediately follows from \eqref{Eq4-omega-nk} and \eqref{Eq-ExplicitDerivative-Example}. 
}
\end{remark}

Finally, we can now state the results that correspond to those given by  Theorem \ref{Thm-QuadRule-UnitCircle}. First we immediately have for any $\mathcal{F}(z) \in span\{z^{-n+1},z^{-n+2},\ldots,z^{n-2},z^{n-1}\}$, 
\begin{equation*} \label{Eq-QR-UC-mu-Example}
   \int_{\T} \mathcal{F}(\zeta)\, d\mu(\zeta) =   \frac{1}{4\pi i} \int_{\T} \mathcal{F}(\zeta)\, \frac{(1-\zeta)(\zeta-1)}{\zeta^2}d \zeta = \sum_{k=1}^{n} \frac{2\sin^2\big(k\pi/(n+1)\big)}{n+1} \, \mathcal{F}(\xi_{n,k}),
\end{equation*}
where $\xi_{n,k}$ are as in \eqref{Eq-zeros-Rn-Example}. 

Observe that one can easily verify that 
\[
   \widehat{\lambda}_{n+1, n+1} = \frac{(1-M_1)(1-M_2) \cdots (1-M_n)}{(1-\ell_1)(1-\ell_2)\cdots(1-\ell_n)} = \frac{1}{n+1}. 
\]
Hence, if $\nu$ is such that 
\begin{equation*} \label{Eq-nu-Class-Example}
      \int_{\T} \phi(\zeta)\,d\nu(\zeta) = \frac{1-\epsilon}{2\pi i}\int_{\T} \phi(\zeta)\,\frac{1}{\zeta} d \zeta + \epsilon\, \phi(1),
\end{equation*}
then for any $\mathcal{F}(z) \in span\{z^{-n},z^{-n+1},\ldots,z^{n-1},z^{n}\}$,
\begin{equation*} \label{Eq-QR-UC-nu-Example-Temp1}
      \int_{\T} \mathcal{F}(\zeta)\, d \nu(\zeta) =  [(1-\epsilon) \frac{1}{n+1} + \epsilon]\, \mathcal{F}(1) + \sum_{k=1}^{n} (1-\epsilon)\frac{1}{n+1} \, \mathcal{F}(\xi_{n,k}).
\end{equation*}
With $\epsilon = 0$ this leads to the well known Gauss-Lebesgue quadrature rule on the unit circle
\begin{equation*} \label{Eq-QR-UC-nu-Example-Temp2}
     \frac{1}{2\pi i} \int_{\T} \mathcal{F}(\zeta)\, \frac{1}{\zeta} d \zeta =  \frac{1}{n+1}  \sum_{k=1}^{n+1}  \mathcal{F}(e^{i2k\pi/(n+1)}) ,
\end{equation*}
valid for  $\mathcal{F}(z) \in span\{z^{-n},z^{-n+1},\ldots,z^{n-1},z^{n}\}$.

\setcounter{equation}{0}
\section{Numerical evaluation of quadrature nodes and weights} \label{Sec-NumEval-NodesWeights}

We now look into some methods for generating the nodes and weights of the quadrature rules given by Theorems \ref{Thm-QuadRule-RealLine} and \ref{Thm-QuadRule-UnitCircle}.  We may consider the respective quadrature rules as companion quadrature rules. That is, if we know the values of the nodes and weights of one of the quadrature rules we also know the corresponding values for the other quadrature rules. 

The two quadrature rules given by Theorem \ref{Thm-QuadRule-UnitCircle} are special cases of interpolatory quadrature rules on the unit circle based on the zeros of para-orthogonal polynomials (see, for example, \cite{Bultheel-MJC-RCB-JCAM2015} and references therein).  We say that they are special cases of such quadrature rules because, in the case of the $n$ - point quadrature rule \eqref{Eq-QR-UC-mu}, the para-orthogonal polynomial involved is $\Psi_{n}(\mu; -\tau_{n}, z) = const \times R_{n}(z)$. Similarly,  in the case of the $n+1$ - point quadrature rule \eqref{Eq-QR-UC-nu}, the para-orthogonal polynomial involved is  $\Psi_{n+1}\big(\nu_{\underline{\epsilon}}; \tau_{n}, z\big) = const \times (z-1) R_{n}(z)$.  One important advantage that we can point out  here is that  $R_{n}$ satisfy the three term recurrence relation  \eqref{Eq-TTRR-Special-ParaOrth-polynomials}. Thus, the  required para-orthogonal polynomials are easily obtained.  

We also recall that the coefficients $\{c_{n}\}_{n \geq 1}$ and $\{d_{n+1}\}_{n \geq 1}$ in \eqref{Eq-TTRR-Special-ParaOrth-polynomials} are exactly those which appear in  Theorem \ref{Thm-Basics} and \eqref{Eq-Verblunsky-Characterization-22}, and how to recover them from the Verblunsky coefficients $\{\alpha_{n}(\mu)\}_{n \geq 0}$ or the Verblunsky coefficients $\{\alpha_{n}(\nu_{\underline{\epsilon}})\}_{n \geq 0}$ can be found, respectively, in Theorem \ref{Thm-Inv-Map-1} and in Theorem \ref{Thm-Inv-Map-2}.   

We may consider methods that are traditionally used to evaluate the nodes  $\xi_{n,k}$ and the corresponding weights (see, for example,  \cite{AmmarCalvettiReichel1996},   \cite{AmmarGraggReichel1988}, \cite{Bultheel-MJC-RCB-JCAM2015}, \cite{Cant-Barro-Vera-2008} and \cite{Watkins-1993}). For these methods the starting point are the eigenvalue problems obtained from a well known unitary modifications of the CMV matrices, respectively,  of degrees $n$ associated with the Verblunsky coefficients $\{\alpha_{n}(\mu)\}_{n \geq 0}$ and of degree $n+1$ associated with the Verblunsky coefficients $\{\alpha_{n}(\nu_{\underline{\epsilon}})\}_{n \geq 0}$.  One draw back in such methods is that one needs the use of complex arithmetic. We will now propose two  methods that require only the use of real arithmetic.

The nodes $x_{n,k}$ of the  quadrature rule given by Theorem \ref{Thm-QuadRule-RealLine} are also the zeros  of $P_{n}$, which  are again  the eigenvalues of the generalized eigenvalue problem given by \eqref{Eq-Special-GEV-problem}. We will  consider these nodes to be arranged as in \eqref{Eq-zero-arrangement}. That is, $x_{n,k+1} < x_{n,k}$, $k =1,2 \ldots, n-1$.  
We will now consider the techniques how these nodes and the corresponding weights in the quadrature rules  can be estimated. The first of these techniques is based on the Laguerre's root finding method. 

\subsection{Laguerre's root finding method (LRF method)}

Use of the Laguerre's root finding method (LRF method) for determining the eigenvalues of real symmetric tridiagonal eigenvalue problems has already been discussed in \cite{WILKINSON-Book} and its efficiency has been proven in \cite{Li-Li-Zeng-NA1994}. The method requires the evaluation of the characteristic polynomial and its first and second derivatives, which can be obtained from the recurrence relations satisfied by these polynomials. For an efficient implementation of the method, a nice way of using these recurrence relations is also considered in \cite{Li-Li-Zeng-NA1994}.

We now explore the method of finding the nodes of the quadrature rule in Theorem \ref{Thm-QuadRule-RealLine} by the LRF method.  Observe that in our case one of the matrices involved in the generalized eigenvalue problem is not a real symmetric matrix, but a complex Hermitian matrix. However,  the characteristic polynomial $P_{n}$ is still real and can be obtained from the very nice three term recurrence relation \eqref{Eq-TTRR-Special-R2-type}.  Moreover, the sequence $\{P_{n}\}_{n \geq 0}$ obtained from this three term recurrence relation has also turned out to be a Sturm sequence (see \cite{Rahman-Schmeisser-2002} and \cite{Schwarz-Stiefel-Rutishauser-1973}), which helps to isolate the zeros of $P_{n}$ within small intervals. 

A direct usage of the LRF method will be as follows:

Let  $y_{0}$ be taken within the interval  $(x_{n,k+1}, x_{n,k})$. Then the iterative process 
\[ 
  \begin{array}{rl}
    y_{j+1} \!\! \! &= L_{+}(y_{j}) \\[0ex]
       &\dsp= y_{j} + \frac{nP_{n}(y_{j})}{-P_{n}^{\prime}(y_{j}) + sign(P_{n}(y_{j}))\sqrt{[(n-1)P_{n}^{\prime}(y_{j})]^2 - n(n-1) P_{n}(y_{j}) P_{n}^{\prime\prime}(y_{j})}},
  \end{array}
\]  
for $j =0,1, \ldots$, converges to $x_{n,k}$. Likewise, the iterative process 
\[ 
  \begin{array}{rl}
    y_{j+1} \!\! \! &= L_{-}(y_{j}) \\[0ex]
      \dsp &\dsp = y_{j} + \frac{nP_{n}(y_{j})}{-P_{n}^{\prime}(y_{j}) - sign(P_{n}(y_{j}))\sqrt{[(n-1)P_{n}^{\prime}(y_{j})]^2 - n(n-1)P_{n}(y_{j})P_{n}^{\prime\prime}(y_{j})}},
  \end{array}
\]  
for $j =0,1, \ldots$, converges to $x_{n,k+1}$. The values of $P_{n}(y_{j})$, $P_{n}^{\prime}(y_{j})$ and $P_{n}^{\prime\prime}(y_{j})$ can be evaluated from  the three term recurrence relation \eqref{Eq-TTRR-Special-R2-type}. 

However, in order to reduce any underflow or overflow problems, we will consider the following modified recurrence relations which are easily obtained  from \eqref{Eq-TTRR-Special-R2-type}:  
\begin{equation} \label{Eq-RecFor-XnYnZn}
  \begin{array}{rl}
      X_{m+1}(x)\!\!\!\!&\dsp = \frac{x-c_{m+1}}{\sqrt{x^2+1}} X_{m}(x) - d_{m+1} X_{m-1}(x), \\[3ex]
      Y_{m+1}(x)\!\!\!\! &\dsp = \frac{x-c_{m+1}}{\sqrt{x^2+1}} Y_{m}(x) - d_{m+1} Y_{m-1}(x) + X_{m}(x) - \frac{2x d_{m+1}}{\sqrt{x^2+1}}X_{m-1}(x), \\[3ex]
      Z_{m+1}(x)\!\!\!\! &\dsp = \frac{x-c_{m+1}}{\sqrt{x^2+1}} Z_{m}(x) - d_{m+1} Z_{m-1}(x) \\[1.5ex]
      & \dsp \hspace{15ex} + \ 2 Y_{m}(x) -\frac{4x d_{m+1}}{\sqrt{x^2+1}}Y_{m-1}(x) -2 d_{m+1} X_{m-1}(x), 
  \end{array}
\end{equation}
for $m = 1,2, \ldots n-1$, where 
\begin{equation} \label{Eq-Pn-to-XnYnZn}
 X_{m}(x) = \frac{P_{m}(x)}{(x^2+1)^{m/2}}, \ \ Y_{m}(x) = \frac{P_{m}^{\prime}(x)}{(x^2+1)^{(m-1)/2}} \ \ \mbox{and} \ \ Z_{m}(x) = \frac{P_{m}^{\prime\prime}(x)}{(x^2+1)^{(m-2)/2}}. 
\end{equation}

In terms of the above modified functions and also with $\mathfrak{S}_{n,k} = sign(X_{n}(y_{k}))$ the LRF method becomes:  
\begin{equation} \label{Eq-IncreasingLaguerre}
  \begin{array}{l}
    y_{j+1} 
       \dsp = y_{j} + \frac{nX_{n}(y_{j})\sqrt{y_{j}^2+1}}{-Y_{n}(y_{j}) + \mathfrak{S}_{n,k}\sqrt{[(n-1)Y_n(y_{j})]^2 - n(n-1)X_{n}(y_{j})Z_{n}(y_{j})}},
  \end{array}
\end{equation}  
for $j =0,1, \ldots$, which converges to $x_{n,k}$ if $y_0$ is chosen between $x_{n,k+1}$ and $x_{n,k}$;  
\begin{equation} \label{Eq-DecreasingLaguerre}
  \begin{array}{l}
    y_{j+1}  
      \dsp = y_{j} + \frac{nX_{n}(y_{j})\sqrt{y_{j}^2+1}}{-Y_{n}(y_{j}) - \mathfrak{S}_{n,k}\sqrt{[(n-1)Y_n(y_{j})]^2 - n(n-1)X_{n}(y_{j})Z_{n}(y_{j})}},
  \end{array}
\end{equation}  
for $j =0,1, \ldots$, which converges to $x_{n,k}$ if $y_0$ is chosen between $x_{n,k}$ and $x_{n,k-1}$.

In order to have a systematic way to approximate  all the zeros of $P_{n}$, first we determine the total number of positive zeros and the total number of negative zeros of $P_{n}$ by {\em the method of Sturm sequence} (see, for example, \cite{Rahman-Schmeisser-2002} and \cite{Schwarz-Stiefel-Rutishauser-1973}). This follows from the number of sign changes within the sequence $\{P_{m}(0)\}_{m=0}^{n}$, or equivalently, within the sequence $\{X_{m}(0)\}_{m=0}^{n}$. Then the approximations for the positive zeros are obtained, in an increasing order,  by \eqref{Eq-IncreasingLaguerre}. Similarly, the  approximations for the   negative zeros are obtained, in a decreasing order, by \eqref{Eq-DecreasingLaguerre}. 

For an efficient convergence of the algorithm given by \eqref{Eq-IncreasingLaguerre} and \eqref{Eq-DecreasingLaguerre}, it is also important the choice of the initial value $y_{0}$ that lead to the convergence of these algorithms  to  a particular $x_{n,k}$. Thus, for example in the case of \eqref{Eq-IncreasingLaguerre}, having determined two or more of the positive zeros nearest to the origin, to determine the next positive zero, say $x_{n,k}$, we use as initial value $y_{0} = x_{n,k+1} + (x_{n,k+1} - x_{n,k+2})\times \delta$.  Clearly, to expect a better convergence it is important that we choose the positive value of $\delta$ so that  $y_{0}$ remains between $x_{n,k+1}$ and $x_{n,k}$, but also much closer to $x_{n,k}$ than $x_{n,k+1}$. With a particular choice of $\delta$ if $y_{0}$ passes the value of $x_{n,k}$, which can be verified by the method of Sturm sequence, we can then assume a new $y_{0}$ obtained with a reduced  value of $\delta$. In general, we have observed that the distance between two consecutive zeros of those lie away from the origin is larger than  the distance between two consecutive zeros which lie closer to the origin. Thus, in most of the examples that we have considered below, choosing $\delta = 1$ seems to have been adequate and has worked very well.   

Similarly, to determine a negative zero, say $x_{n, k}$, we use as initial value in \eqref{Eq-DecreasingLaguerre}  $y_{0} = x_{n,k-1} - (x_{n,k-2} - x_{n,k-1})\times \delta$, with $\delta$ chosen using the same idea as before. 

Having found the required approximate value of a zero $x_{n,k}$, the corresponding approximate value of the weights $\omega_{n,k}$ can be obtained from
\begin{equation} \label{Eq22-omega-nk}
    \omega_{n,k} =  \frac{d_2d_3\cdots d_nM_1}{Y_n(x_{n,k})X_{n-1}(x_{n,k})}, \quad k = 1,2, \ldots, n. 
\end{equation}
This follows from \eqref{Eq1-omega-nk} and \eqref{Eq-Pn-to-XnYnZn}.
Observe that the values of $Y_n(x_{n,k})$ and $X_{n-1}(x_{n,k})$ are readily available from the final stage of convergence to $x_{n,k}$ in LRF method. 

As follows from Theorem \ref{Thm-Basics}, the value of $M_1$ can be obtained from the formulas 
\[
    M_1 = (c_1^2+1)\int_{-\infty}^{\infty} \frac{1}{x^2 +1}\,d\varphi(x) = (c_1^2+1)\int_{\T} \frac{(\zeta-1)^2}{\zeta^2}\,d\nu_{\underline{0}}(\zeta),
\]
where the probability measure $\nu_{\underline{0}}$ on the unit circle is such that $d \nu_{\underline{0}}\big((x+i)/(x-i)\big) = -d \varphi(x)$. If we know the Verblunsky coefficients associated with $\nu_{\underline{0}}$  then from Theorem \ref{Thm-Inv-Map-2},  
\[
    M_1 = \frac{1}{2}\, \frac{\big|1 - \alpha_{0}(\nu_{\underline{0}})\big|^2}{\Re\big(1 -\alpha_{0}(\nu_{\underline{0}})\big)}. 
\]

The LRF method is known to have a cubic convergence rate.  Perhaps the drawbacks one could see in using the LRF method are:

\noindent - square root evaluations, which are more time consuming; 
  
\noindent - evaluations of the quantities $X_{n}$, $Y_n$ and $Z_{n}$ by the recursive formulas \eqref{Eq-RecFor-XnYnZn}, which needs careful considerations. 

For informations concerning this latter comment see the paper \cite{Gautschi-SIAMRev-1967} by Gautschi. However, in all the numerical experiments that we have performed, we have not encountered any drawbacks in using the recursive formulas given by \eqref{Eq-RecFor-XnYnZn}.     

\subsection{Inverse power method (IP method)}

We now discuss an alternative way to determine the zeros $x_{n,k}$ of $P_{n}$.  Since these zeros are the eigenvalues of the generalized eigenvalue problem \eqref{Eq-Special-GEV-problem}, we will see how the zero $x_{n,k}$  can be determined from an initial approximation $p$ to this zero  and then the use  of the inverse power method (IP method). The inverse power method is known to be a powerful tool for determining the eigenvectors (see \cite{WILKINSON-Book}). 

In the case of our generalized eigenvalue problem the inverse power method can be stated as
\[  
  \begin{array}{l}
     (\mathbf{A}_n - p \mathbf{B}_n) \mathbf{w}_{n}\bl j \br =\mathbf{B}_n \mathbf{u}_{n}\bl j \br , \\[2ex]
     \dsp \epsilon\bl j \br  = \frac{\mathbf{u}_{n}\bl j \br ^{H}\mathbf{u}_{n}\bl j \br }{\mathbf{u}_{n}\bl j \br ^{H}\mathbf{w}_{n}\bl j \br }, \\[2ex]
     \dsp \mathbf{u}_{n}[j+1] = \gamma\bl j \br \, \mathbf{w}_{n}\bl j \br ,     
  \end{array} \quad j = 0, 1, \ldots ,
\]
where $\mathbf{A}_n$ and $\mathbf{B}_{n}$ are as in \eqref{Eq-Special-GEV-problem}. Here $\gamma\bl j \br $ are  normalization constants.  If $p$ is chosen close enough to $x_{n,k}$  (but not equal to) then $\{\epsilon\bl j \br \}_{j \geq 0}$ converges to $x_{n,k} - p$. 

From \eqref{Eq-Rationals-1}, since the eigenvector 
\[
   \mathbf{u}_{n} = \mathbf{u}_{n}(x_{n,k}) = \big[u_{n,0}(x_{n,k}), u_{n,1}(x_{n,k}), \ldots, u_{n,n-1}(x_{n,k})\big]^{T}
\]
can always be chosen such that the element $u_{n,0}(x_{n,k}) = P_{0}(x_{n,k}) = 1$, we will assume that 
\[
    \mathbf{u}_{n}\bl j \br  = \big[u_{n,0}\bl j \br , u_{n,1}\bl j \br , \ldots, u_{n,n-1}\bl j \br \big]^{T}
\]
is such that  $u_{n,0}\bl j \br  = 1$. Thus, if we set 
\[
    \mathbf{w}_{n}\bl j \br  = \big[w_{n,0}\bl j \br , w_{n,1}\bl j \br , \ldots, w_{n,n-1}\bl j \br \big]^{T}
\]
we must take $\gamma\bl j \br  = 1/w_{n,0}\bl j \br $. With this normalization constant, we also find that $\{\mathbf{u}_{n}\bl j \br \}_{j \geq 0}$ converges to the eigenvector $\mathbf{u}_{n}$ and $\{\gamma\bl j \br \}_{j \geq 0}$ converges to $x_{n,k} - p$.  

However, to determine the value of $\mathbf{w}_{n}\bl j \br $ from $\mathbf{B}_n \mathbf{u}_{n}\bl j \br $, we make  use of LU decomposition. As the matrix $\mathbf{A}_n$  and the eigenvectors $\mathbf{u}_{n}(x_{n,k})$ are complex, one may also expect the necessity for the use of complex arithmetic. However, because of the easy structure of the matrices $\mathbf{A}_n$ and $\mathbf{B}_n$ we can easily perform all the operations with real arithmetic. 

It turned out, by considering $\mathbf{A}_n - p \mathbf{B}_n = \mathbf{L}_{n}\mathbf{U}_{n}$,    where 
\[
    \mathbf{L}_n= \left[ 
       \begin{array}{cccccc}
        1  &  0 & 0  &  \cdots & 0 & 0  \\[1ex]
        l_{n,0}  & 1  & 0 &  \cdots & 0 & 0 \\[1ex]
         0  & l_{n,1}  & 1 &  \cdots & 0 & 0  \\[1ex]
         \vdots  & \vdots  & \vdots &    & \vdots & \vdots \\[1ex]
          0  & 0  & 0 &  \cdots & 1 &  0  \\[1ex]
          0  & 0  & 0 &  \cdots & l_{n,n-2} & 1 
      \end{array} \! \right], 
      \quad 
       \mathbf{U}_n = \left[ 
       \begin{array}{cccccc}
        \mathfrak{r}_{n,0}  &  t_{n,0} & 0  &  \cdots & 0 & 0  \\[1ex]
        0  & \mathfrak{r}_{n,1}  & t_{n,1} &  \cdots & 0 & 0 \\[1ex]
         0  & 0  & \mathfrak{r}_{n,2} &  \cdots & 0 & 0  \\[1ex]
         \vdots  & \vdots  & \vdots &    & \vdots & \vdots \\[1ex]
          0  & 0  & 0 &  \cdots & \mathfrak{r}_{n,n-2} &  t_{n,n-2}  \\[1ex]
          0  & 0  & 0 &  \cdots & 0 & \mathfrak{r}_{n,n-1} 
      \end{array} \! \right],       
\]
 that  the elements $l_{n,m}$ and $t_{n,m}$ are complex and the elements $\mathfrak{r}_{n,m}$ are real. Precisely, by setting $l_{n,m} = l_{n,m}^{(1)} + i\, l_{n,m}^{(2)}$ and $\mathfrak{t}_{n,m} = \mathfrak{t}_{n,m}^{(1)} + i\, \mathfrak{t}_{n,m}^{(2)}$, we can state the following.

\begin{theorem} The elements of the matrices $\mathbf{L}_{n}$ and $\mathbf{U}_{n}$ satisfy 
\[
  \begin{array}{l}
     \dsp \ \mathfrak{r}_{n,m} = - \frac{P_{m+1}(p)}{P_{m}(p)}, \quad m = 0, 1, \ldots, n-1,  \\[3ex]
     \begin{array} {ll}
       t_{n,m}^{(1)} = -p\sqrt{d_{m+2}}, &  \ t_{n,m}^{(2)} = \sqrt{d_{m+2}},  \\[2ex]
       \dsp l_{n,m}^{(1)} = \frac{t_{n,m}^{(1)}}{\mathfrak{r}_{n,m}}, & \dsp \ l_{n,m}^{(2)} = \frac{t_{n,m}^{(2)}}{\mathfrak{r}_{n,m}}, \\[2ex]
      \end{array}  m = 0, 1, \ldots, n-2.
  \end{array}
\]
\end{theorem}

\begin{proof}
It is not difficult verify from  $\mathbf{L}_{n}\mathbf{U}_{n} = \mathbf{A}_n - p \mathbf{B}_n$ that $\mathfrak{r}_{n,0} = c_1 - p$, 
\[
  \begin{array}{l}
     \mathfrak{t}_{n,m}^{(1)} = -p\sqrt{d_{m+2}}, \qquad \ \mathfrak{t}_{n,m}^{(2)} = \sqrt{d_{m+2}},  \\[2ex]
     l_{n,m}^{(1)} = -p\sqrt{d_{m+2}}/\mathfrak{r}_{n,m}, \quad l_{n,m}^{(2)} = -\sqrt{d_{m+2}}/\mathfrak{r}_{n,m}, \\[2ex]
     \mathfrak{r}_{n,m+1} = c_{m+2} - p - (p^2+1) d_{m+2}/\mathfrak{r}_{n,m},
  \end{array}
\]
for $m = 0, 1, \ldots, n-2$. Now, we verify from the three term recurrence \eqref{Eq-TTRR-Special-R2-type} that $-\mathfrak{r}_{n,m} = P_{m+1}(p)/P_{m}(p)$, $m \geq 0$. 
\end{proof}

By also setting $\widehat{\mathbf{u}}_{n}\bl j \br  = \mathbf{B}_n \mathbf{u}_{n}\bl j \br $, $\mathbf{L}_n \mathbf{v}_{n}\bl j \br  = \widehat{\mathbf{u}}_{n}\bl j \br $ and $\mathbf{U}_n \mathbf{w}_{n}\bl j \br  = \mathbf{v}_{n}\bl j \br $, where 
\[
  \begin{array}l
    \mathbf{u}_{n}\bl j \br   = \big[u_{n,0}^{(1)}\bl j \br +i\,u_{n,0}^{(2)}\bl j \br , \,u_{n,1}^{(1)}\bl j \br +i\,u_{n,1}^{(2)}\bl j \br , \ldots,\, u_{n,n-1}^{(1)}\bl j \br +i\,u_{n,n-1}^{(2)}\bl j \br \big]^{T}, \\[1.5ex]
    \widehat{\mathbf{u}}_{n}\bl j \br   = \big[\widehat{u}_{n,0}^{(1)}\bl j \br +i\,\widehat{u}_{n,0}^{(2)}\bl j \br , \,\widehat{u}_{n,1}^{(1)}\bl j \br +i\,\widehat{u}_{n,1}^{(2)}\bl j \br , \ldots,\, \widehat{u}_{n,n-1}^{(1)}\bl j \br +i\,\widehat{u}_{n,n-1}^{(2)}\bl j \br \big]^{T}, \\[1.5ex]
    \mathbf{v}_{n}\bl j \br   = \big[v_{n,0}^{(1)}\bl j \br +i\,v_{n,0}^{(2)}\bl j \br , \,v_{n,1}^{(1)}\bl j \br +i\,v_{n,1}^{(2)}\bl j \br , \ldots,\, v_{n,n-1}^{(1)}\bl j \br +i\,v_{n,n-1}^{(2)}\bl j \br \big]^{T}, \\[1.5ex]
    \mathbf{w}_{n}\bl j \br   = \big[w_{n,0}^{(1)}\bl j \br +i\,w_{n,0}^{(2)}\bl j \br , \,w_{n,1}^{(1)}\bl j \br +i\,w_{n,1}^{(2)}\bl j \br , \ldots,\, w_{n,n-1}^{(1)}\bl j \br +i\,w_{n,n-1}^{(2)}\bl j \br \big]^{T}, 
  \end{array}
\] 
one finds: 
\[
   \widehat{u}_{n,0}^{(1)}\bl j \br  = u_{n,0}^{(1)}\bl j \br  + \sqrt{d_{2}}\,u_{n,1}^{(1)}\bl j \br , \quad \widehat{u}_{n,0}^{(2)}\bl j \br  = u_{n,0}^{(2)}\bl j \br  + \sqrt{d_{2}}\,u_{n,1}^{(2)}\bl j \br ,
\]
\[
  \begin{array}{l}
    \widehat{u}_{n,m}^{(1)}\bl j \br  = \sqrt{d_{m+1}}\,u_{n,m-1}^{(1)}\bl j \br  +u_{n,m}^{(1)}\bl j \br  + \sqrt{d_{m+2}}\,u_{n,m+1}^{(1)}\bl j \br , \\[1.5ex]
    \widehat{u}_{n,m}^{(2)}\bl j \br  = \sqrt{d_{m+1}}\,u_{n,m-1}^{(2)}\bl j \br  +u_{n,m}^{(2)}\bl j \br  + \sqrt{d_{m+2}}\,u_{n,m+1}^{(2)}\bl j \br , 
  \end{array} \ \ m =1,2, \ldots, n-2,
\]
\[
   \widehat{u}_{n,n-1}^{(1)}\bl j \br  = \sqrt{d_{n}}\,u_{n,n-2}^{(1)}\bl j \br  +u_{n,n-1}^{(1)}\bl j \br ,  \quad \widehat{u}_{n,n-1}^{(2)}\bl j \br  = \sqrt{d_{n}}\,u_{n,n-2}^{(2)}\bl j \br  +u_{n,n-1}^{(2)}\bl j \br ;
\]
\vspace{0.5ex}
\[
   v_{n,0}^{(1)}\bl j \br  = \widehat{u}_{n,0}^{(1)}\bl j \br , \quad v_{n,0}^{(2)}\bl j \br  = \widehat{u}_{n,0}^{(2)}\bl j \br ,
\]
\[ 
   \begin{array}{l}
     v_{n,m}^{(1)}\bl j \br = \widehat{u}_{n,m}^{(1)}\bl j \br  - l_{n,m-1}^{(1)}v_{n,m-1}^{(1)}\bl j \br  + l_{n,m-1}^{(2)}v_{n,m-1}^{(2)}\bl j \br , \\[1.5ex]
     v_{n,m}^{(2)}\bl j \br = \widehat{u}_{n,m}^{(2)}\bl j \br  - l_{n,m-1}^{(2)}v_{n,m-1}^{(1)}\bl j \br  - l_{n,m-1}^{(1)}v_{n,m-1}^{(2)}\bl j \br , 
   \end{array} \ \ m = 1,2, \ldots, n-1; 
\]
\vspace{1.5ex}
\[
   w_{n,n-1}^{(1)}\bl j \br  = v_{n,n-1}^{(1)}\bl j \br /\mathfrak{r}_{n,n-1}, \quad w_{n,n-1}^{(2)}\bl j \br  = v_{n,n-1}^{(2)}\bl j \br /\mathfrak{r}_{n,n-1}, 
\]
\[ 
   \begin{array}{l}
     w_{n,m}^{(1)}\bl j \br = \big(v_{n,m}^{(1)}\bl j \br  - \mathfrak{t}_{n,m}^{(1)}w_{n,m+1}^{(1)}\bl j \br  + \mathfrak{t}_{n,m}^{(2)}w_{n,m+1}^{(2)}\bl j \br \big)/\mathfrak{r}_{n,m}, \\[1ex]
     w_{n,m}^{(2)}\bl j \br = \big(v_{n,m}^{(2)}\bl j \br  - \mathfrak{t}_{n,m}^{(2)}w_{n,m+1}^{(1)}\bl j \br  - \mathfrak{t}_{n,m}^{(1)} w_{n,m+1}^{(2)}\bl j \br \big)/\mathfrak{r}_{n,m}, 
   \end{array} \ \ m = n-2,n-3, \ldots, 0.
\]
Finally, $u_{n,0}^{(1)}\bl j+1 \br = 1$, $u_{n,0}^{(2)}\bl j+1 \br = 0$ and  
\[
  \begin{array}l
    \Re(\gamma\bl j \br ) = \big(w_{n,0}^{(1)}\bl j \br \big)/\big((w_{n,0}^{(1)}\bl j \br )^2+ (w_{n,0}^{(2)}\bl j \br )^2\big), \\[1ex]
    u_{n,m}^{(1)}\bl j+1\br = \big(w_{n,m}^{(1)}\bl j \br w_{n,0}^{(1)}\bl j \br  + w_{n,m}^{(2)}\bl j \br w_{n,0}^{(2)}\bl j \br \big)/\big((w_{n,0}^{(1)}\bl j \br )^2+ (w_{n,0}^{(2)}\bl j \br )^2\big), \\[1ex]
    u_{n,m}^{(2)}\bl j+1\br = \big(w_{n,m}^{(2)}\bl j \br w_{n,0}^{(1)}\bl j \br  - w_{n,m}^{(1)}\bl j \br w_{n,0}^{(2)}\bl j \br \big)/\big((w_{n,0}^{(1)}\bl j \br )^2+ (w_{n,0}^{(2)}\bl j \br )^2\big), 
  \end{array}
\]
for $m = 1,2, \ldots, n-1$. 

Since the eigenvalues are real and simple, one has (see, for example, \cite{WILKINSON-Book})
\[
    \lim_{j \to \infty} \Re(\gamma\bl j \br ) = x_{n,k} - p \quad \mbox{and} \quad \lim_{j \to \infty} \frac{M_1}{ \mathbf{u}_{n}\bl j \br ^{H}  \widehat{\mathbf{u}}_{n}\bl j \br } \to \omega_{n,k}.
\]
The latter expression follows from \eqref{Eq1-omega-nk}. 

Once we start with the values of $\sqrt{d_{n+1}}$ and $P_{n}(p)$, $n \geq 1$, this method no longer requires the evaluations of square roots.  However, the convergence in this case is linear, in contrast with the LRF method, which is known to be cubic. The closer the value of $p$ to the zero $x_{n,k}$ the faster the convergence. However, $p$ can not be too close to $x_{n,k}$, as  in this case  the system $\mathbf{A}_n - p \mathbf{B}_n$ becomes more ill conditioned.  

In our applications of the IP method in one of the examples stated below, we have used the LRF method to obtain the initial approximation $p$.

\subsection{Numerical examples} \label{subSec-NumEx}

We will consider the numerical evaluation of the nodes and weight of the quadrature rule given by Theorem \ref{Thm-QuadRule-RealLine} for two specific values of $\varphi$. The starting point for our numerical calculations are the associated sequences $\{c_{n}\}_{n \geq 1}$ and $\{d_{n+1}\}_{n \geq 1}$.  All the arithmetic is performed in double precision with the use of the programming language {\em Python}.

\begin{example} \label{Eg-1}
{\em We first look at the numerical evaluation of the nodes and weights of the $n$ - point quadrature rule given by Theorem \ref{Thm-QuadRule-RealLine}, in the case of the sequences  $\{c_{n}\}_{n \geq 1}$ and $\{d_{n+1}\}_{n \geq 1}$ considered in Section \ref{Sec-Eample-simple}.  
}
\end{example} 
The associated probability measure $\nu_{\underline{0}}$ is the Lebesgue measure. 
We also have $d \varphi(x) = 1/[\pi(x^2+1)]$ and $\int_{-\infty}^{\infty} d \varphi(x)  = 1$.
This is a nice test example since the nodes $x_{n,k}$ and the weights $\omega_{n,k}$ are explicitly given as in \eqref{Eq-zeros-Rn-Example} and \eqref{Eq-QR-RL-Example}. 

In Table\,\ref{table1}  we have given the results that we obtain for the nodes of the $15$-point quadrature rule using the LRF method.  Since the nodes (i.e., zeros of $P_{15}$) are symmetric about the origin, we evaluate only the positive zeros using \eqref{Eq-IncreasingLaguerre}.  As we have mentioned earlier, these positive zeros  are evaluated in an increasing order of magnitude. Clearly, $x_{15,8} = 0$ and the first zero that we need to evaluate  is  $x_{15,7}$. Accept for the initial approximation $y_0 = 0.1$ to arrive at the value of $x_{15,7}$, for any of the remaining positive zeros $x_{15,k}$, $k = 6, 5, \ldots, 1$, the initial approximation $y_0$ is taken to be $x_{n,k+1} + (x_{n,k+1} - x_{n,k+2})$. Also with  each of these $x_{15,k}$ the process of determining its approximations of $y_j$, $j=0, 1, \ldots$,  from  \eqref{Eq-IncreasingLaguerre} is repeated until the difference between two successive approximations $y_{j}$ and  $y_{j-1}$ becomes less than $10^{-10}$.    The third column of Table\,\ref{table1} gives the number of iterations $j$ (i.e., the value of $j$ when $|y_{j}- y_{j-1}|$ becomes smaller than $10^{-10}$). 

\begin{table}[H]
\centering
\begin{tabular}{c|c|c|c|c}
\hline
& $y_0$   & \ \ $j$ \ \  & $y_j$ & zeros with 15 digits \\ 
\hline
\ $x_{15,7}$ \ & \hspace{1ex} $0.1000000000$ \hspace{1ex} & $5$ & \hspace{1ex} $0.1989123673$ \hspace{1ex} & \ \hspace{1ex} $      0.19891236737966\ldots$ \hspace{1ex} \\
\hline
$x_{15,6}$ & $0.3978247347$ & $3$ & $0.4142135623$ & \ $         0.41421356237310\ldots$\\
\hline
$x_{15,5}$ & $0.6295147573$ & $4$ & $0.6681786379$ & \ $         0.66817863791930\ldots$\\
\hline
$x_{15,4}$ & $0.9221437134$ & $4$ & $1.0000000000$ & \ $         1.00000000000000\ldots$\\
\hline
$x_{15,3}$ & $1.3318213620$ & $4$ & $1.4966057626$ & \ $         1.49660576266549\ldots$\\
\hline
$x_{15,2}$ & $1.9932115253$ & $4$ & $2.4142135623$ & \ $         2.41421356237309\ldots$\\
\hline
$x_{15,1}$ & $3.3318213620$ & $5$ & $5.0273394921$ & \ $         5.02733949212585\ldots$\\
\hline
\end{tabular}
\caption{Positive zeros by using LRF method for Example \ref{Eg-1}, with $n =15$, where $j$ is the number of iterations to achieve the precision $10^{-10}$.}
\label{table1}
\end{table}

The results given in column $4$ of Table\,\ref{table1} represent the approximate values for the zeros obtained after the specified number of iterations. Comparing with the exact values, with 15 digits, given in column $5$ of Table\,\ref{table1}, the approximations presented in column $4$ are exactly the same as the exact values for all 11 digits presented.   

\begin{table}[H]
\centering 
\begin{tabular}{c|c|c|c|c|c}
\hline
& $y_0$ & \ \ $j$ \ \ & \ \ $k$ \ \ & $y_k$ & zeros with 15 digits     \\ 
\hline
\ $x_{15,7}$ \ &\hspace{1ex}  $0.10000000000$ \hspace{1ex} & $3$ & $3$ & \hspace{1ex} $0.1989123673$ \hspace{1ex} & \hspace{1ex} $ 0.19891236737966\ldots$ \hspace{1ex}\\
\hline
$x_{15,6}$ & $0.39782473476$ & $2$ & $3$ & $0.4142135623$ & \ $         0.41421356237310\ldots$\\
\hline
$x_{15,5}$ & $0.62951475737$ & $2$ & $3$ & $0.6681786379$ & \ $         0.66817863791930\ldots$\\
\hline
$x_{15,4}$ & $0.92214371347$ & $2$ & $3$ & $1.0000000000$ & \ $         1.00000000000000\ldots$\\
\hline
$x_{15,3}$ & $1.33182136208$ & $3$ & $1$ & $1.4966057626$ & \ $         1.49660576266549\ldots$\\
\hline
$x_{15,2}$ & $1.99321152533$ & $3$ & $3$ & $2.4142135623$ & \ $         2.41421356237309\ldots$\\
\hline
$x_{15,1}$ & $3.33182136208$ & $4$ & $1$ & $5.0273394921$ & \ $         5.02733949212585\ldots$\\
\hline
\end{tabular}
\caption{Positives zeros by using hybrid LRF-IP method for Example \ref{Eg-1} with $n =15$, where $j$ is the number of iterations of LRF method to achieve the precision $10^{-2}$  and $k$ is the number of iterations of IP method to achieve the precision $10^{-10}$.}
\label{table2}
\end{table}

Results given in Table\,\ref{table2} are what we obtain with IP method with real arithmetic. For the record, with the use of complex arithmetic (not always readily available) no difference in accuracy has been observed.  To obtain an approximation for a zero  $x_{15,k}$ using the IP method, we first need the initial approximation $p$. In the results presented in Table\,\ref{table2} we arrive at this initial approximation $p$  with the use of LRF method. Instead of using  $10^{-10}$ to stop the iterations with the LRF method we have used  $10^{-2}$. We then perform the iteration with the IP method until the difference between two successive iterations becomes smaller than $10^{-10}$.  We will refer to  this technique of starting with the LRF method and finishing with the IP method {\em the hybrid LRF-IP method}.  Even though the IP method is a nice and simple alternative method and the final achievements are the same,  comparing the results obtained in Tables \ref{table1} and \ref{table2} no significant advantages in terms of convergence over the LRF method is observed, especially for the approximations of zeros closer to the origin.

\begin{example} \label{Eg-2}
{\em We now consider the numerical evaluation of the nodes and weights of the quadrature rules that follow from the three term recurrence \eqref{Eq-TTRR-Special-R2-type}, where   
\[
   c_{n} = c_{n}^{(b)} = \frac{\eta}{\lambda+n}, \quad d_{n+1} = d_{n+1}^{(b)} = \frac{1}{4} \frac{n(n+2\lambda+1)}{(n+\lambda)(n+\lambda+1)}, \quad n \geq 1. 
\]
Here, $\eta \in \mathbb{R}$ and $\lambda > -1/2$. 
}
\end{example}

The polynomials $\{P_{n}\}_{n \geq 0} = \{P_{n}^{(b)}\}_{n \geq 0}$ obtained here have been referred to as complementary Romanovski-Routh  polynomials (see \cite{AMF-LLSR-ASR-MT-PAMS2019} and references therein).  They satisfy the orthogonality given by \eqref{Eq-Orthogonality-for-Pn}, with  
\begin{equation} \label{Eq-Routh-Romanovski-measure}
   d \varphi(x) = d \varphi^{(b)}(x) = \frac{e^{\pi \Im(b)}}{2\pi}  \frac{2^{b+\bar{b}+1} |\Gamma(b+1)|^2}{\Gamma(b+\bar{b}+1)}   \frac{(e^{-{\rm arccot}\,x})^{2\Im(b)}}{(x^2+1)^{\Re(b)+1}} dx, 
\end{equation}  
where $b = \lambda + i \eta$. The probability measures on the unit circle obtained as in Theorem \ref{Thm-Basics} are 
\[
  \begin{array}{l}    
    d \nu_{\underline{0}}(e^{i \theta}) = d \nu^{(b)}(e^{i \theta}) = \frac{2^{b+\bar{b}}|\Gamma(b+1)|^2}{2\pi\Gamma(b+\bar{b}+1)}e^{(\pi-\theta)\Im(b)}[sin^2(\theta/2)]^{\Re(b)}d\theta, \\[2ex]
    d\mu(e^{i \theta}) = d \nu^{(b+1)}(e^{i \theta}).   \\[2ex]
  \end{array} 
\]
These can also be written in the following equivalent forms
\begin{equation} \label{CRR-Measures}
  \begin{array}{l}
    d \nu_{\underline{0}}(\zeta) = d \nu^{(b)}(\zeta) = \tau(b)  \zeta^{-(\bar{b}+1)}(\zeta-1)^{2\Re(b)}d\zeta \quad \mbox{and} \quad d\mu(\zeta) = d \nu^{(b+1)}(\zeta), 
  \end{array} 
\end{equation}
where 
\[
      \tau(b) = \frac{|\Gamma(b+1)|^2}{2 \pi i\Gamma(b+\bar{b}+1)}\frac{e^{\pi \Im(b)}}{(-1)^{\Re(b)}}. 
\]

The minimal $\{\ell_{n+1}^{(b)}\}_{n \geq 0}$ and maximal $\{M_{n+1}^{(b)}\}_{n \geq 0}$ parameter sequences of the positive chain sequence $\{d_{n+1}^{(b)}\}_{n \geq 1}$ are 
\[
      \ell_{n+1}^{(b)} = \frac{n}{2(n+\lambda+1)}  \quad   \mbox{and} \quad    M_{n+1}^{(b)} = \frac{n + 2\lambda + 1}{2(n+\lambda+1)},  \quad n \geq 0 .
\]
For the above measure $\nu_{\underline{0}} = \nu^{(b)}$ if we consider the $(n+1)$-point quadrature rule on the unit circle given by Theorem \ref{Thm-QuadRule-UnitCircle}, we have for $\mathcal{F}(z) \in span\{z^{-n},z^{-n+1},\ldots,z^{n-1},z^{n}\}$,
\begin{equation} \label{Eq-QR-UC-nu-temp}
       \int_{\T} \mathcal{F}(\zeta)\, d \nu^{(b)}(\zeta) =  \widehat{\lambda}_{n+1,n+1}^{(b)}\, \mathcal{F}(1) + \sum_{k=1}^{n}\omega_{n,k}^{(b)} \, \mathcal{F}(\xi_{n,k}^{(b)}),
\end{equation}
where $\xi_{n,k}^{(b)} = (x_{n,k}+i)/(x_{n,k}-i) $  and 
\[
       \widehat{\lambda}_{n+1,n+1}^{(b)} = \frac{(1-M_1^{(b)})(1-M_2^{(b)}) \cdots (1-M_n^{(b)})}{(1-\ell_1^{(b)})(1-\ell_2^{(b)})\cdots(1-\ell_n^{(b)})} = \frac{n!}{(2\lambda+2)_{n}}. 
\]
Here,  $x_{n,k}^{(b)}$ (the zeros of $P_{n}^{(b)}$) and $\omega_{n,k}^{(b)}$ are the nodes and weights of the $n$-point quadrature rule given by Theorem \ref{Thm-QuadRule-RealLine} with $\varphi = \varphi^{(b)}$ is as in \eqref{Eq-Routh-Romanovski-measure}.

\begin{table}[H]
\centering 
\begin{tabular}{c|r|c|r|r}
\hline
\ \ $k$ \ \ & $y_0$ \hspace{4ex} & \ \ $j$ \ \ &  $x_{8,k}^{(b)}$ \hspace{5ex} & $\omega_{8,k}^{(b)}$  \hspace{4ex}  \\ 
\hline
$8$ & \hspace{1ex} $-0.71588$ \hspace{1ex} &  $4$ & \hspace{1ex} $ -0.860951902$ \hspace{1ex} & \hspace{1ex} $ 0.001435559$\\
\hline
$7$ & $-0.20000$ \hspace{1ex} &  $4$ & $ -0.395455713$ \hspace{1ex} & $              0.013120781$\\
\hline
$6$ & $ 0.00000$ \hspace{1ex} &  $4$ & $ -0.075029910$ \hspace{1ex} & $              0.057779655$\\
\hline
$5$ & $ 0.00000$ \hspace{1ex} &  $6$ & $  0.211994598$ \hspace{1ex} & $              0.155038062$\\
\hline
$4$ & $ 0.30000$ \hspace{1ex} &  $3$ & $  0.519849212$ \hspace{1ex} & $              0.268406695$\\
\hline
$3$ &  $0.82770$ \hspace{1ex} &  $4$ & $  0.909786866$ \hspace{1ex} & $              0.291154810$\\
\hline
$2$ &  $1.29972$ \hspace{1ex} &  $4$ & $  1.509028782$ \hspace{1ex} & $              0.173690345$\\
\hline
$1$ & $2.10827$ \hspace{1ex} &  $4$ & $  2.752206638$ \hspace{1ex} & $              0.039041093$\\
\hline
\end{tabular}
\caption{Nodes and weights of the quadrature rule of Theorem  \ref{Thm-QuadRule-RealLine} for Example \ref{Eg-2} with $n=8$ and  $b = 2.5+i 2.0$,  where $y_0$ is the initial approximation for the respective zero and $j$ is the number of iterations of LRF method to achieve the precision $10^{-10}$.}
\label{table3} 
\end{table} 

\begin{table}[H]
\centering 
\begin{tabular}{c|r|c|r|r}
\hline
\ \ $k$ \ \ & $y_0$ \hspace{4ex} & \ \ $j$ \ \ &  $x_{15,k}^{(b)}$ \hspace{5ex} & $\omega_{15,k}^{(b)}$  \hspace{4ex}  \\ 
\hline
$15$ & \hspace{1ex}  $-1.41686$ \hspace{1ex} &  $4$ & \hspace{1ex}  $             -1.672044257$ \hspace{1ex} & \hspace{1ex} $  0.000036057$\\
\hline
$14$ &                  $-0.96894$ \hspace{1ex} &  $4$ &   $             -1.066959532$ \hspace{1ex} & $              0.000311365$\\
\hline
$13$ &                  $-0.67016$ \hspace{1ex} &  $4$ &   $             -0.717060414$ \hspace{1ex} & $              0.001519919$\\
\hline
$12$ &                  $-0.44218$ \hspace{1ex} &  $4$ &   $             -0.465177200$ \hspace{1ex} & $              0.005341485$\\
\hline
$11$ &                  $-0.15000$ \hspace{1ex} &  $4$ &   $             -0.260191665$ \hspace{1ex} & $              0.014845009$\\
\hline
$10$ &                   $0.00000$ \hspace{1ex} &  $5$ &   $             -0.078205917$ \hspace{1ex} & $              0.034128298$\\
\hline
$9$ &                   $0.00000$ \hspace{1ex} &  $5$ &   $              0.095146340$ \hspace{1ex} & $              0.066361078$\\
\hline
$8$ &                   $0.15000$ \hspace{1ex} &  $4$ &   $              0.270925228$ \hspace{1ex} & $              0.110088169$\\
\hline
$7$ &                   $0.44670$ \hspace{1ex} &  $3$ &   $              0.460151608$ \hspace{1ex} & $              0.155554797$\\
\hline
$6$ &                   $0.64938$ \hspace{1ex} &  $4$ &   $              0.676720369$ \hspace{1ex} & $              0.185015149$\\
\hline
$5$ &                   $0.89329$ \hspace{1ex} &  $4$ &   $              0.941766842$ \hspace{1ex} & $              0.180719442$\\
\hline
$4$ &                   $1.20681$ \hspace{1ex} &  $4$ &   $              1.292753697$ \hspace{1ex} & $              0.138672540$\\
\hline
$3$ &                   $1.64374$ \hspace{1ex} &  $4$ &   $              1.807020312$ \hspace{1ex} & $              0.077127555$\\
\hline
$2$ &                   $2.32129$ \hspace{1ex} &  $4$ &   $              2.679413438$ \hspace{1ex} & $              0.026491638$\\
\hline
$1$ &                   $3.55181$ \hspace{1ex} &  $4$ &   $              4.607169720$ \hspace{1ex} & $              0.003769069$\\
\hline
\end{tabular}
\caption{Nodes and weights of the quadrature rule of Theorem  \ref{Thm-QuadRule-RealLine} for Example \ref{Eg-2} with $n=15$ and  $b = 2.5+i 2.0$,  where $y_0$ is the initial approximation for the respective zero and $j$ is the number of iterations of LRF method to achieve the precision $10^{-10}$.}
\label{table4} 
\end{table} 

\begin{table}[H]
\centering 
\begin{tabular}{c|r|c|r|r}
\hline
\ \ $k$ \ \ & $y_0$ \hspace{4ex} & \ \ $j$ \ \  &  $x_{8,k}^{(b)}$ \hspace{5ex} & $\omega_{8,k}^{(b)}$  \hspace{4ex}  \\ 
\hline
$8$ &                  $-0.71475$ \hspace{1ex} &  $4$ &   $             -0.866362671$ \hspace{1ex} & $              0.001409047$\\
\hline
$7$ & \hspace{1ex}  $-0.20000$ \hspace{1ex} &  $4$ &  \hspace{1ex} $             -0.385089950$ \hspace{1ex} & \hspace{1ex} $  0.011285827$\\
\hline
$6$ &                   $0.00000$ \hspace{1ex} &  $4$ &   $             -0.055426036$ \hspace{1ex} & $              0.047818577$\\
\hline
$5$ &                   $0.00000$ \hspace{1ex} &  $6$ &   $              0.242186897$ \hspace{1ex} & $              0.131133033$\\
\hline
$4$ &                   $0.30000$ \hspace{1ex} &  $3$ &   $              0.567035907$ \hspace{1ex} & $              0.243675010$\\
\hline
$3$ &                   $0.89188$ \hspace{1ex} &  $4$ &   $              0.990130503$ \hspace{1ex} & $              0.296947815$\\
\hline
$2$ &                   $1.41323$ \hspace{1ex} &  $4$ &   $              1.668212121$ \hspace{1ex} & $              0.208595193$\\
\hline
$1$ &                   $2.34629$ \hspace{1ex} &  $4$ &   $              3.172646563$ \hspace{1ex} & $              0.058358497$\\
\hline
\end{tabular}
\caption{Nodes and weights of the quadrature rule of Theorem  \ref{Thm-QuadRule-RealLine} for Example \ref{Eg-2} with $n=8$ and  $b = 2.0+i 2.0$,  where $y_0$ is the initial approximation for the respective zero and $j$ is the number of iterations of LRF method to achieve the precision $10^{-10}$.}
\label{table5}
\end{table}

\begin{table}[H]
\centering 
\begin{tabular}{c|r|c|r|r}
\hline
\ \  $k$ \ \  & $y_0$ \hspace{4ex} & \ \ $j$ \ \  &  $x_{15,k}^{(b)}$ \hspace{5ex} & $\omega_{15,k}^{(b)}$ \hspace{4ex}    \\ 
\hline
$15$ & \hspace{1ex} $-1.43682$ \hspace{1ex} &  $4$ &  \hspace{1ex}  $             -1.709139557$ \hspace{1ex} & \hspace{1ex} $ 0.000047582$ \\
\hline
$14$ &                  $-0.97423$ \hspace{1ex} &  $4$ &   $             -1.076874551$ \hspace{1ex} & $              0.000329605$\\
\hline
$13$ &                  $-0.66851$ \hspace{1ex} &  $4$ &   $             -0.716927042$ \hspace{1ex} & $              0.001407835$\\
\hline
$12$ &                  $-0.43632$ \hspace{1ex} &  $4$ &   $             -0.459623282$ \hspace{1ex} & $              0.004551300$\\
\hline
$11$ &                  $-0.15000$ \hspace{1ex} &  $4$ &   $             -0.250739572$ \hspace{1ex} & $              0.012058883$\\
\hline
$10$ &                   $0.00000$ \hspace{1ex} &  $4$ &   $             -0.065156395$ \hspace{1ex} & $              0.027196733$\\
\hline
$9$ &                   $0.00000$ \hspace{1ex} &  $5$ &   $              0.112221377$ \hspace{1ex} & $              0.053180697$\\
\hline
$8$ &                   $0.15000$ \hspace{1ex} &  $4$ &   $              0.293142015$ \hspace{1ex} & $              0.090767684$\\
\hline
$7$ &                   $0.47406$ \hspace{1ex} &  $3$ &   $              0.489563410$ \hspace{1ex} & $              0.134906482$\\
\hline
$6$ &                   $0.68598$ \hspace{1ex} &  $4$ &   $              0.716961300$ \hspace{1ex} & $              0.172611607$\\
\hline
$5$ &                   $0.94436$ \hspace{1ex} &  $4$ &   $              0.999518459$ \hspace{1ex} & $              0.185747261$\\
\hline
$4$ &                   $1.28208$ \hspace{1ex} &  $4$ &   $              1.381314327$ \hspace{1ex} & $              0.161215301$\\
\hline
$3$ &                   $1.76311$ \hspace{1ex} &  $4$ &   $              1.956293085$ \hspace{1ex} & $              0.104560922$\\
\hline
$2$ &                   $2.53127$ \hspace{1ex} &  $4$ &   $              2.970861856$ \hspace{1ex} & $              0.043473255$\\
\hline
$1$ &                   $3.98543$ \hspace{1ex} &  $4$ &   $              5.358584571$ \hspace{1ex} & $              0.007880354$\\
\hline
\end{tabular}
\caption{Nodes and weights of the quadrature rule of Theorem  \ref{Thm-QuadRule-RealLine} for Example \ref{Eg-2} with $n=15$ and  $b = 2.0 +i 2.0$,  where $y_0$ is the initial approximation for the respective zero and $j$ is the number of iterations of LRF method to achieve the precision $10^{-10}$.}
\label{table6} 
\end{table} 

Tables \ref{table3}, \ref{table4}, \ref{table5} and \ref{table6}  give the results that we have obtained for the associated $8$-point and $15$-point quadrature rules of Theorem  \ref{Thm-QuadRule-RealLine}, with two different choices of $\lambda$ and $\eta$.  The values of the nodes $x_{n,k}^{(b)}$ are found using the LRF method. The results given in the last column of these table are the corresponding values of $\omega_{n,k}^{(b)}$, which we have derived with the use of \eqref{Eq22-omega-nk}. The results found in these tables will be used in Section \ref{Sec-Applications-of-QR} to determine the values of certain integrals. 

The results given in  Table\,\ref{table3} and Table\,\ref{table4} are those corresponding to  the choice $b =  2.5 + i 2.0$. The results in these tables were obtained as follows. Once the number of positive zeros and negative zeros of $P_{n}^{(b)}$ were determined by the method of Sturm sequence,  the positive zeros were successively derived using the recursive process \eqref{Eq-IncreasingLaguerre} and  the  negative zeros were then successively derived using the recursive process  \eqref{Eq-DecreasingLaguerre}. The process of determining approximations  $y_j$, $j =0,1, \ldots$,  to a $x_{n,k}^{(b)}$ using  \eqref{Eq-IncreasingLaguerre}, or \eqref{Eq-DecreasingLaguerre}, is repeated until the difference between two successive approximations $y_{j}$ and  $y_{j-1}$ becomes less than $10^{-10}$.  

The results given in  Table \ref{table5} and Table \ref{table6}, obtained similarly,  are those corresponding to the choice  $b = 2.0 + i2.0$. 

We inform that all the final approximations found in Tables \ref{table3}, \ref{table4}, \ref{table5} and \ref{table6} are  correct to all the digits presented.

\setcounter{equation}{0}
\section{Applications of the quadrature rules}
\label{Sec-Applications-of-QR}

In this section we consider some applications of the quadrature rules that we have considered in Subsection \ref{subSec-NumEx} to see their convergence to specific integrals.

The first application is with respect to the $n$-point quadrature rule on the real line given by \eqref{Eq-QR-RL-Example}. The weights in this case take the exact values $\omega_{n,k} =1/(n+1)$. Numerical results that we have obtained with respect to the evaluation of the nodes of the associated $15$-point quadrature rule are given by Example \ref{Eg-1}.  It is important to observe that when $f$ is a  polynomial of any degree (including constant), the $n$-point quadrature sum does not equals the  exact value of the integral. However, in Example\,\ref{Eg-3} below, we see that the quadrature sums provide very good approximations for the associated integral. 

\begin{example} \label{Eg-3}  {\em Consider the estimation of the integral $I = \int_{-\infty}^{\infty} (x^2+1)^{-8} e^{-x^2} dx$. 
}
\end{example}

The exact value of this integral for $13$ significant digits is $0.6133229495946$.  A possible choice of an $n$-point quadrature rule for the estimation of this integral is the $n$-point Gauss-Hermite rule. If  
\[ 
GH_{n} = \sum_{k=1}^{n} \hat{\omega}_{n,k} \frac{1}{(\hat{x}_{n,k}^2+1)^{8}}
\]
is the $n$-point Gauss-Hermite quadrature sum for the function $(x^2+1)^{-8}$, then one could expect that $GH_{n}$ converges to $I$.  The results obtained with these quadrature sums and error are given in columns $4$ and $5$ of Table \ref{table7}, respectively.  Clearly, the results are not  very satisfactory.  

However, with the quadrature rule of Theorem  \ref{Thm-QuadRule-RealLine}  we can consider also the possibility of letting 
\[
     I =  \int_{-\infty}^{\infty}  \pi(x^2+1)^{-7}e^{-x^2} \frac{1}{\pi(x^2+1)}dx = \int_{-\infty}^{\infty} \pi(x^2+1)^{-7}e^{-x^2} d \varphi(x) 
\]
and then to use the quadrature sum
\[
    I_n = \frac{\pi}{n+1}  \sum_{k=1}^{n} (x_{n,k}^2+1)^{-7}e^{-x_{n,k}^2}
\]
to estimate the above integral. The exact values of $x_{n,k}$ in the above sum are those found in Section \ref{Sec-Eample-simple}.  Results that we have obtained here and the error are presented as columns $2$ and $3$ in Table \ref{table7}, respectively. 

Even though $(x^2+1)^{n}f(x) = \pi (x^2+1)^{n-7}e^{-x^2} \notin \mathbb{P}_{2n-1}$,  the quadrature sums $I_n$ converge rapidly to $I$. 
\begin{table}[H]
\centering 
\begin{tabular}{c|c|c|c|c}
\hline 
$n$ & $I_{n}$  & $|I-I_{n}|$   & $GH_{n}$ & $|I-GH_{n}|$ \\ 
\hline 
$6$ \ & \ $0.61228678065306$ \ & \ $1.0e-03$ \ & \ $0.36007906855948$ \ & \ $2.5e-01$\\ 
\hline 
$10$ \ & $0.61332311526782$ & $1.6e-07$ & $0.50344321854055$ &$1.0e-01$\\ 
\hline 
$12$ & $0.61332296550298$ & $1.5e-08$ & $0.53991060640781$ & $7.3e-02$\\ 
\hline 
$15$ & $0.61332294881837$ & $    7.7e-10$ & $0.65430343845086$ & $4.0e-02$\\
\hline 
\end{tabular} 

\caption{Application of Gauss-Hermite quadrature rule and quadrature rule (\ref{Eq0-QR-RL}) to the integral considered in Example \ref{Eg-3}.}
\label{table7}
\end{table}

\begin{example} \label{Eg-4}  {\em Consider the  integrals $ \int_{\T}g_1(\zeta) d\zeta$ and $ \int_{\T}g_2(\zeta) d\zeta$, where  
}
\[
     g_1(\zeta) = \sin(\zeta)\, \zeta^{-2.5+i2.0} \frac{(\zeta-1)^5 }{(4-\zeta)} \quad \mbox{and} \quad g_2(\zeta) = \sin(\zeta)\, \zeta^{-2.0+i2.0} \frac{(\zeta-1)^5 }{(4-\zeta)}.
\]
\end{example}

The exact values of the above integrals are found to be 
\[
    S_1 = \int_{\T}g_1(\zeta) d\zeta = (3.52677323641868\ldots + i \, 2.86020606590488\ldots)\times 10^{-2}
\]
and
\[
   S_2 =  \int_{\T}g_2(\zeta) d\zeta = (0.33606707423377\ldots + i \, 2.80064202619193\ldots)\times 10^{-2}.
\]

Now we look into the numerical estimations of the above integrals with the use of the $(n+1)$-point quadrature rule (\ref{Eq-QR-UC-nu}) given by Theorem \ref{Thm-QuadRule-UnitCircle}. Clearly, 
\[
      \int_{\T}g_1(\zeta) d\zeta = \frac{1}{\tau(2.5+i2.0)} \int_{\T} \mathcal{F}_1(\zeta)\, d \nu^{(2.5+i2.0)}(\zeta)
\]
and
\[
      \int_{\T}g_2(\zeta) d\zeta = \frac{1}{\tau(2.5+i2.0)} \int_{\T} \mathcal{F}_2(\zeta)\, d \nu^{(2.5+i2.0)}(\zeta),
\]
where $\nu^{(2.5+i2.0)}$ and $\tau(2.5+i2.0)$ are as in \eqref{CRR-Measures},  and 
\[
     \mathcal{F}_1(\zeta) = \frac{\zeta \sin(\zeta)}{4-\zeta} \quad \mbox{and} \quad \mathcal{F}_2(\zeta) = \frac{\zeta^{1.5} \sin(\zeta)}{4-\zeta}. 
\]
Thus, we can estimate the integrals $\int_{\T} \mathcal{F}_1(\zeta)\, d \nu^{(2.5+i2.0)}(\zeta)$ and $\int_{\T} \mathcal{F}_2(\zeta)\, d \nu^{(2.5+i2.0)}(\zeta)$, using the nodes and weights obtained in Tables \ref{table3} and \ref{table4}. With the knowledge that
\[
     \tau(2.5+i2.0) = -2.26887229599887\ldots,
\]
the results obtained with the associated $(8+1)$-point and $(15+1)$-point quadrature rules given by  \eqref{Eq-QR-UC-nu-temp} are in Tables \ref{table8} and \ref{table9}. 

\begin{table}[H]
\centering
\begin{tabular}{c|c|c}
\hline 
($n+1$) & approximation for $S_1$ & absolute error   \\ 
\hline 
$(8+1)$&$     3.52677470437557e-02+  i   2.86021897172897e-02$ &$                   1.3e-07$\\
\hline  
$(15+1)$&$     3.52677323654955e-02+  i   2.86020606599670e-02$ &$                   1.6e-12$\\ 
\hline 
\end{tabular}
\caption{Application of the $(n+1)$-point quadrature rule in $\frac{1}{\tau(2.5+i2.0)} \int_{\T} \mathcal{F}_1(\zeta)\, d \nu^{(2.5+i2.0)}(\zeta)$.}
\label{table8}
\end{table}

\begin{table}[H]
\centering
\begin{tabular}{c|c|c}
\hline 
($n+1$) & approximation for $S_2$ & absolute error  \\ 
\hline 
$(8+1)$&$  3.36088971644750e-03+   i 2.80129450967408e-02$ &$                   6.5e-06$\\ 
\hline  
$(15+1)$&$ 3.36059488687229e-03+   i 2.80065722184178e-02$ &$            1.7e-07$\\
\hline 
\end{tabular}
\caption{Application of the $(n+1)$-point quadrature rule in $\frac{1}{\tau(2.5+i2.0)} \int_{\T} \mathcal{F}_2(\zeta)\, d \nu^{(2.5+i2.0)}(\zeta)$.}
\label{table9}
\end{table}

However, the results obtained for $S_2$ is not very satisfactory. This we believe is because of the component $\zeta^{1/2}$ in $\mathcal{F}_2(\zeta)$.  We can remedy this by considering 
\[
    S_2 =  \int_{\T}g_2(\zeta) d\zeta = \frac{1}{\tau(2.0+i2.0)} \int_{\T} \hat{\mathcal{F}}_2(\zeta)\, d \nu^{(2.0+i2.0)}(\zeta), 
\]
where  
\[
    \hat{\mathcal{F}}_2(\zeta) = (\zeta -1) \frac{\zeta \sin(\zeta)}{4-\zeta}.  
\]
Thus, we can estimate the integral  $\int_{\T} \hat{\mathcal{F}}_2(\zeta)\, d \nu^{(2.0+i2.0)}(\zeta)/\tau(2.0+i2.0)$, see Table \ref{table10}, using the nodes and weights obtained in Tables \ref{table5} and \ref{table6}. The convergence is again very good. 

\begin{table}[H]
\centering 
\begin{tabular}{c|c|c}
\hline 
($n+1$) & approximation for $S_2$ & absolute error  \\ 
\hline 
$(8+1)$ &  $ 3.36106005666877e-03+    i 2.80065917218239e-02$ &$                   4.3e-07$\\  
\hline  
$(15+1)$ & $ 3.36067074166006e-03+   i  2.80064202570487e-02$ &$                   4.9e-12$\\ 
\hline 
\end{tabular}
\caption{Application of the $(n+1)$-point quadrature rule in $\frac{1}{\tau(2.0+i2.0)} \int_{\T} \hat{\mathcal{F}}_2(\zeta)\, d \nu^{(2.0+i2.0)}(\zeta)$.}
\label{table10}
\end{table}


\begin{thebibliography}{15}

\bibitem{AmmarCalvettiReichel1996} 
G.S. Ammar, D. Calvetti and L. Reichel, Continuation methods for the computation of zeros of Szeg\H{o} polynomials, {\em Linear Algebra Appl.}, 249 (1996), 125-155.


\bibitem{AmmarGraggReichel1988} 
G. Ammar, W. Gragg and L. Reichel, Constructing a unitary Hessenberg matrix from spectral data, Numerical Linear Algebra, Digital Signal
Processing and Parallel Algorithms, Leuven, 1988, pp. 385-395, NATO Advanced Science Institutes Series F: Computer and Systems Sci.,
vol. 70.

\bibitem{BracSilvaRanga-AMC2016} 
C.F. Bracciali, J.S. Silva and A. Sri Ranga, Explicit formulas for OPUC and POPUC associated with measures which are simple modifications of the Lebesgue measure, \emph{Appl. Math. Comput.}, 271 (2015), 820-831.

\bibitem{BSRS_2016}  
C.F. Bracciali, A. Sri Ranga and A. Swaminathan, Para-orthogonal polynomials on the unit circle satisfying three term recurrence formulas, \emph{Appl. Numer. Math.}, 19 (2016), 19-40.

\bibitem{Bultheel-MJC-RCB-JCAM2015} 
A. Bultheel, M.J. Cantero and R. Cruz-Barroso, Matrix methods for quadrature formulas on the unit circle. A survey, {\em J. Comput. Appl. Math.}, 284 (2015), 78-100. 

\bibitem{Cant-Barro-Vera-2008} 
M.J. Cantero, R. Cruz-Barroso and P. Gonz\'{a}lez-Vera, A matrix approach to the computation of quadrature formulas on the unit circle, \emph{Appl. Numer. Math.}, 58 (2008), 296-318.

\bibitem{Chihara-Book}
T.S. Chihara, \emph{An Introduction to Orthogonal Polynomials}, Mathematics and its Application Series, Gordon and Breach, New York, 1978.

\bibitem{Costa-Felix-Ranga-JAT2013}
 M.S. Costa, H.M. Felix and A. Sri Ranga, Orthogonal polynomials on the unit circle and chain sequences, {\em J. Approx. Theory}, 173 (2013), 14-32.

\bibitem{Gautschi-SIAMRev-1967} 
W. Gautschi, Computational aspects of three-term recurrence relations, {SIAM Review}, 9 (1967), 24-82. 

\bibitem{Gautschi-1981} 
W. Gautschi, A survey of Gauss-Christoffel quadrature formulae, in: P.L. Butzer, F. Feh\'{e}r (Eds.), E.B. Christoffel: ``The Influence of his Work on Mathematics and the Physical Sciences'', Birkh\"{a}user, Basel, 1981, pp. 72-147. 

\bibitem{Gautschi-SiamJSSC82}  
W. Gautschi, On generating orthogonal polynomials, {\em  SIAM J. Sci. Statist. Comput.}, 3 (1982), 289–317.

\bibitem{ENZG1} 
T. Erd\'elyi, P. Nevai, J. Zhang and J. Geronimo, A simple proof of ``Favard's theorem'' on the unit circle, {\em Atti Sem. Mat. Fis. Univ. Modena}, 39  (1991), 551--556.  Also in ``Trends in Functional Analysis and Approximation Theory'' (Acquafredda di Maratea, 1989), 41--46, Univ. Modena Reggio Emilia, Modena, 1991. 

\bibitem{Ismail-1989} 
M.E.H. Ismail, Monotonicity of zeros of orthogonal polynomials, in: D.Stanton, (Ed.), ``q-Series and Partitions'', Springer-Verlag, New York, 1989, pp. 177-190. 

\bibitem{Ismail-Masson-JAT1995}
M.E.H. Ismail and D.R. Masson, Generalized orthogonality and continued fractions, \emph{J. Approx. Theory}, 83 (1995), 1-40.

\bibitem{Ismail-Ranga}  
M.E.H. Ismail and A. Sri Ranga,  $R_{II}$  type recurrence, generalized eigenvalue problem and orthogonal polynomials on the unit circle, {\em Linear Algebra  Appl.}, 562 (2019), 63-90.

\bibitem{JONES} 
W.B. Jones, O. Nj\aa stad and W.J. Thron, Moment theory, orthogonal polynomials, quadrature and continued fractions associated with the unit circle, \emph{Bull. London Math. Soc.}, 21 (1989) 113-152. 

\bibitem{Li-Li-Zeng-NA1994}
K. Li, T.Y. Li and Z.Zeng,  An algorithm for the generalized symmetric tridigonal eigenvalue problem, {\em Numer. Algorithms}, 8 (1994), 269-291.

\bibitem{AMF-LLSR-ASR-MT-PAMS2019} A. Mart\'{i}nez-Finkelshtein, L.L. Silva Ribeiro,  A. Sri Ranga  and M. Tyaglov, Complementary Romanovski-Routh polynomials: From orthogonal polynomials on the unit circle to Coulomb wave functions, {\em Proc. Amer. Math. Soc.}, to appear. 

\bibitem{Rahman-Schmeisser-2002} 
Q.I. Rahman and G. Schmeisser, \emph{Analytic Theory of Polynomials}, Clarendon Press, Oxford,  2002.

\bibitem{Schwarz-Stiefel-Rutishauser-1973}
H.R. Schwarz, E. Stiefel and H. Rutishauser \emph{Numerical Analysis of Symmetric Matrices}, Prentice-Hall Press, Englewood Cliffs, NJ,  1973.

\bibitem{Simon-book-p1} 
B. Simon, {``Orthogonal Polynomials on the Unit Circle. Part 1. Classical Theory''},  American Mathematical Society Colloquium Publications, vol. {\bf  54}, part 1, (American Mathematical Society, Providence, RI,  2004).

\bibitem{Verblunsky-1935} 
S. Verblunsky, On positive harmonic functions: a contribution to the algebra of Fourier series, \emph{Proc. London Math. Soc.}, 38 (1935), 125-157. 

\bibitem{Watkins-1993} 
D.S. Watkins, Some perspectives on the eigenvalue problem, {\em SIAM Rev.},  35 (1993), 430-471.

\bibitem{WILKINSON-Book}
J.H. Wilkinson, \emph{The Algebraic Eigenvalue Problem}, Oxford University Press, Oxford, 1965.

\bibitem{Zhedanov-JAT1999} 
A. Zhedanov, Biorthogonal rational functions and generalized eigenvalue problem, {\em J. Approx. Theory}, 101 (1999), 303-329.

\end{thebibliography}
\end{document}